\documentclass[mathpazo]{cicp}

\usepackage[left=2.5cm,right=2.5cm,top=2cm,bottom=2cm]{geometry}
\usepackage{amsmath,amssymb,amsfonts,amsthm,bm}
\usepackage{xcolor,comment}
\usepackage{soul}
\soulregister\ref7
\soulregister\citealt7
\usepackage{enumerate}
\usepackage{hyperref}
\usepackage[ruled,lined,shortend,linesnumberedhidden]{algorithm2e}
\usepackage{graphicx}
\usepackage[section]{placeins}
\usepackage{epstopdf}
\usepackage{framed}
\usepackage{multirow}
\usepackage{float}
\usepackage[caption=false]{subfig}
\usepackage{makecell}
\newtheorem{assumption}{Assumption}

\numberwithin{equation}{section}

\begin{document}

\title{Accelerated Nehari manifold optimization methods for computing ground state solutions of coupled elliptic systems}

\author[Z.~Chen, W.~Liu, Z.~Xie, W.~Yi]{%
    Zhaoxing Chen\affil{1}\comma\affil{2},
    Wei Liu\affil{3}\comma\corrauth, 
    Ziqing Xie\affil{1}, and 
    Wenfan Yi\affil{4}}
\address{%
    \affilnum{1}\ {MOE-LCSM, School of Mathematics and Statistics, Institute of Interdisciplinary Studies, Hunan Normal University, Changsha 410081, China.} \\
    \affilnum{2}\ {Institute of Natural Sciences, Shanghai Jiao Tong University, Shanghai 200240, China.} \\
    \affilnum{3}\ {College of Science, National University of Defense Technology, Changsha 410073, China.} \\
    \affilnum{4}\ {Hunan Provincial Key Laboratory of Intelligent Information Processing and Applied Mathematics, School of Mathematics, Hunan University, Changsha 410082, China.}}
\emails{{\tt {helloczx@sjtu.edu.cn}} (Z.~Chen), 
{\tt {wl@nudt.edu.cn}} (W.~Liu),
{\tt {ziqingxie@hunnu.edu.cn}} (Z.~Xie),
{\tt {wfyi@hnu.edu.cn}} (W.~Yi)}

\begin{abstract}
This paper presents a class of efficient  algorithms for computing the ground state solutions of a semilinear elliptic system, which  are unstable saddle points of the variational functional. Variational arguments show that these unstable saddle points can be characterized as the local minimizers of the  variational functional constrained to the Nehari manifold $\mathcal{N}$. The Nehari manifold optimization method (NMOM) proposed in [Z. Chen, W. Liu, Z. Xie, and W. Yi. {\em SIAM J. Sci. Comput.}, 47(4): A2098-A2126, 2025] provides a Riemannian gradient descent framework on $\mathcal{N}$  for such constrained minimization problems. To  deal with both the intrinsic instability of the solutions and the increased computational complexity introduced by the coupling between components, we combine the  ideas from the NMOM  and the Nesterov-type acceleration to develop a new  efficient Riemannian accelerated gradient algorithm on $\mathcal{N}$ (RAG-$\mathcal{N}$). The  key idea is to perform an easy-to-implement nonlinear extrapolation step on $\mathcal{N}$, followed by a Riemannian steepest-descent update at the extrapolated point. To enhance the robustness, we further incorporate a nonmonotone step-size search strategy into the RAG-$\mathcal{N}$ algorithm,  obtaining a  variant with improved stability. Numerical experiments show that the RAG-$\mathcal{N}$  algorithms substantially reduce the number of iterations compared with the Riemannian steepest descent algorithm of NMOM. Finally, we apply the RAG-$\mathcal{N}$ algorithms to compute the ground state solutions of semilinear elliptic systems with two, three and four components, and investigate their behavior under different coupling  coefficients and  various settings. Specifically,  a phase diagram separating fully nontrivial and semi-trivial ground state solutions in the coupling-parameter space is numerically obtained.
\end{abstract}


\maketitle


\section{Introduction}
Let $H = H_0^1(\Omega,\mathbb{R}^m)$ be the Sobolev space, and consider the following $m$-coupled semilinear elliptic system for $\mathbf{u} = (u_1,u_2,\ldots,u_m)\in H$ ~\cite{DANCER2010953,Sato2015,SOAVE2015,YuBo2020}
\begin{align}\label{eq:model}
\left\{\begin{aligned}
&-\varepsilon_i\Delta u_i(\mathbf{x}) + a_i(\mathbf{x}) u_i(\mathbf{x}) =\sum_{j=1}^m  g_{ij}(u_j(\mathbf{x}))^2 u_i(\mathbf{x}),\quad \mathbf{x} \in \Omega, \\
&u_i(\mathbf{x}) = 0,\quad \mathbf{x} \in \partial \Omega,
\end{aligned}\right. \quad i=1,2,\ldots,m.
\end{align}
Here, $\Omega\subset\mathbb{R}^N$ ($N=1,2,3$) is a bounded domain with Lipschitz boundary $\partial\Omega$ and the parameters $\varepsilon_i\in \mathbb{R}^+$ ($1\leq i\leq m$) are diffusion coefficients. The coupling constants $g_{ij}$ and functions $a_i(\mathbf{x})$ ($1\leq i,j\leq m$) are given to satisfy the following basic assumptions such that the associated variational functional admits a suitable variational structure.
\begin{assumption}\label{ass:A}
Assume that (i) the matrix $\mathbf{g} :=  \left(g_{ij}\right)_{i,j=1}^m\in \mathbb{R}^{m\times m}$ is symmetric, i.e., $g_{ij}=g_{ji}$, and satisfies either the {\em fully cooperative condition:} $g_{ij}>0$ $\forall\,1\leq i,j\leq m$, or the {\em positively coupled condition:} $\mathbf{g}$ is positive-definite, and (ii) the function ${\bf a}:=(a_1(\mathbf{x}),a_2(\mathbf{x}),\ldots,a_m(\mathbf{x}))\in L^{\infty}(\Omega,\mathbb{R}^m)$ satisfies $a_i(\mathbf{x})\geq a_i^0 >- \varepsilon_i \lambda_1\;\;\mathrm{a.e.}\;\;\mathbf{x}\in\Omega,\;\;\forall\,\;1\leq i\leq m$, where $a_i^0$ are some constants and  $\lambda_1>0$ denotes the first Dirichlet eigenvalue of $-\Delta$ in $\Omega$. 
\end{assumption}

The coupled system \eqref{eq:model} serves as a fundamental model in broad scientific fields such as quantum mechanics, nonlinear optics, and materials science. In the scalar case of $m=1$, the coupled system~\eqref{eq:model} reduces to a single-component one, which has been extensively studied (see~\cite{CLXY2025SISC,Rabinowitz1986} and references therein). When multiple components are present and interact through nonlinear coupling, the coupled system~\eqref{eq:model} exhibits a variety of phenomena that do not occur in the single-equation setting, making it particularly relevant in applications such as nonlinear optics~\cite{manakov1974theory,shabat1972exact,Sirakov2007} and multi-component Bose-Einstein condensates (BECs)~\cite{Bao2004,2013BaoBEC,Bao-Du2004}. Actually, the coupled system~\eqref{eq:model} arises naturally in the literature for the special stationary state solutions $\psi_i(\mathbf{x},t) := u_i(\mathbf{x})\exp({\mathrm{i}\omega_i t})$ of the following coupled nonlinear Schr\"odinger (NLS) system under truncation on the bounded domain $\Omega$ with Dirichlet boundary conditions:
 \begin{align}\label{eq:NLSE}
     \mathrm{i} \partial_{t} \psi_i(\mathbf{x},t) = -\varepsilon_i\Delta \psi_i(\mathbf{x},t) + V_i(\mathbf{x})\psi_i(\mathbf{x},t) -\sum_{j=1}^m g_{ij}|\psi_j(\mathbf{x},t)|^2\psi_i(\mathbf{x},t).
 \end{align}
Here $\mathrm{i}=\sqrt{-1}$ is the imaginary unit, $\omega_i\in\mathbb{R}$ are the prescribed frequencies (or chemical potentials) and $V_i(\mathbf{x}):=a_i(\mathbf{x})-\omega_i$ represent the external potentials. Particularly, when the external potential is absent ($V_i =0$ for all $i = 1, 2, \ldots, m$), the coupled NLS system~\eqref{eq:NLSE} is widely applied in nonlinear optics, modeling the propagation of $m$ self-trapped, mutually incoherent wave packets in photo-refractive media with a drift-type nonlinear response~\cite{PRL_Akhmediev1999}. In this context, the solution $\boldsymbol{\psi} =(\psi_1,\psi_2,\ldots,\psi_m)$ represents the components of the optical beam in Kerr-like photo-refractive media~\cite{manakov1974theory,shabat1972exact,Sirakov2007}. When $V_i$, $i = 1, 2, \ldots, m$,  are present and chosen as box potentials, harmonic potentials or lattice potentials, the coupled NLS system~\eqref{eq:NLSE} also models multi-component BECs trapped in external potentials~\cite{Bao2004,2013BaoBEC,Bao-Du2004}. Physically, the solution $\boldsymbol{\psi}$ denotes the macroscopic wave function. The positive constants $g_{jj} > 0$ represent the strengths of the self-focusing interaction in the $j$-th component of the beam or the condensate, while the coupling constants $g_{ij}$ with $i\neq j$ model the interaction between the $i$-th and $j$-th components. If $g_{ij}>0$, the interaction is said to be attractive; if $g_{ij}<0$, it is repulsive. Due to the wide application background, the coupled system~\eqref{eq:model} has attracted considerable attention, with theoretical works~\cite{DANCER2010953,Sato2015,SOAVE2015} and the numerical works~\cite{CHEN2010,CHEN2008,YuBo2020}.

The variational functional corresponding to the coupled system~\eqref{eq:model} is
\begin{align}\label{eq:energy}
E(\mathbf{u}) = \frac{1}{2}\int_{\Omega}\sum_{i=1}^m\left(\varepsilon_i|\nabla u_i(\mathbf{x})|^2 + a_i(\mathbf{x})|u_i(\mathbf{x})|^2 \right)d\mathbf{x} - \frac{1}{4} \int_{\Omega}\sum_{i,j=1}^{m}   g_{ij}|u_i(\mathbf{x})|^2 |u_j(\mathbf{x})|^2 d \mathbf{x}. 
\end{align}
Clearly, $E\in C^2(H,\mathbb{R})$. Weak solutions of the coupled system~\eqref{eq:model} correspond to critical points $\mathbf{u^*}$ of $E$ in $H$, i.e. $E'(\mathbf{u}^*) = 0$, where $E'(\mathbf{u}^*)$ denotes the Fr\'echet derivative of $E$ at $\mathbf{u}^*$. It is well known that the coupled system~\eqref{eq:model} may admit multiple solutions~\cite{DANCER2010953,Rabinowitz1986,Sato2015}. Among all nonzero solutions in $H$, we focus on a type of special nonzero solutions $\mathbf{u}_g$ that minimize the variational functional $E$. Such nonzero solutions $\mathbf{u}_g$ are often referred to as the ground state solutions~\cite{Berestycki1983I,Djairo1998}. Since the term ``ground state solution'' may have different interpretations in different literature, following \cite{Berestycki1983I,CORREIA2016,CORREIA2016JFA,Djairo1998,LYZ2023SISC}, we introduce the following definition for clarity. 
\begin{definition}\label{def:ug}
$\mathbf{u}_g\in H\backslash\{0\}$ is called a {\em ground state solution} of the coupled system~\eqref{eq:model}, if 
\[ E'(\mathbf{u}_g)=0\quad\mbox{and}\quad E(\mathbf{u}_g)=\min\left\{E(\mathbf{u})\;\big|\; \mathbf{u}\in H\backslash\{0\},\, E'(\mathbf{u})=0\right\}. \]
\end{definition}

From the variational viewpoint~\cite{1994Infinite}, a ground state solution $\mathbf{u}_g$ of the coupled system~\eqref{eq:model} is an unstable critical point of $E$ with the Morse index (MI) equal to one (1-saddle), i.e. the maximal dimension of negative definite subspace of the Hessian operator $E''(\mathbf{u}_g)$ in $H$ is one. Equivalently, at $\mathbf{u}_g=(u_1^g,u_2^g,\ldots,u_m^g)$, the eigenvalue problem of the linearized operator of the system \eqref{eq:model} (after a sign change), i.e.,
\begin{align}\label{eq:spec_steady-state1}
 	 	\left\{\begin{aligned}
 		   & \lambda w_i(\mathbf{x}) = \varepsilon_i \Delta w_i(\mathbf{x}) -a_i(\mathbf{x}) {w}_i(\mathbf{x}) + \sum_{j=1}^m g_{ij}\left( |u_j^g(\mathbf{x})|^2 w_i(\mathbf{x}) +  2u_i^g(\mathbf{x})u_j^g(\mathbf{x}) {w}_j(\mathbf{x})\right), \quad \mathbf{x} \in \Omega,   \\ 
 		 &   w_i(\mathbf{x})  =0,\quad \mathbf{x} \in {\partial \Omega}, \,\quad  i=1,2,\ldots,m,
 		 \end{aligned} \right.
\end{align}
has exactly one positive eigenvalue $\lambda >0$.
By the linear stability theory~\cite{ChenHu2008,Kielhofer1974}, this implies that $\mathbf{u}_g$ is also a linearly unstable steady state of the following $L^2$-gradient flow, 
\begin{align}\label{eq:model-RD}
\left\{\begin{aligned}
&\frac{\partial \phi_i(\mathbf{x},t)}{\partial t}=\varepsilon_i\Delta \phi_i(\mathbf{x},t) - a_i(\mathbf{x}) \phi_i(\mathbf{x},t) +\sum_{j=1}^m  g_{ij}|\phi_j(\mathbf{x},t)|^2 \phi_i(\mathbf{x},t),\quad \mathbf{x} \in \Omega,\quad t>0,  \\
&\phi_i(\mathbf{x},t)  = 0, \quad \mathbf{x} \in \partial \Omega, \quad i=1,2,\ldots,m,
\end{aligned}\right. 
\end{align}
in the sense that any small perturbation along the eigenfunction direction corresponding to the positive eigenvalue $\lambda$ drives the trajectory away from $\mathbf{u}_g$ exponentially. Therefore, one cannot stably compute $\mathbf{u}_g$ by simply evolving \eqref{eq:model-RD} from an initial guess. The combination of the multiplicity and instability of the solutions and the complexity introduced by the coupling between components makes the efficient numerical computation of ground state solutions of the coupled system~\eqref{eq:model} highly challenging.

In critical point theory, the Nehari manifold  \cite{1960On,1961Characteristic,2010Themethod}  
\begin{align}\label{eq:Nehari manifold-1}
\mathcal{N} 
&= \left\{\mathbf{u} \in H\backslash\{0\}:  \langle E'(\mathbf{u}), \mathbf{u}\rangle = 0\right\}
\end{align}
is introduced to study the ground state critical points of functional $E$, where $\langle \cdot, \cdot \rangle$ denotes the duality pairing between $H$ and its dual space $H^*$. Clearly, $\mathcal{N}$ contains all nontrivial critical points of $E$. Recently, by exploiting the geometry of the Nehari manifold $\mathcal{N}$, a Nehari manifold optimization method (NMOM) is developed to find 1-saddles of a generic functional by minimizing the same functional on the associated Nehari manifold \cite{CLXY2025SISC}. In fact, the Nehari manifold $\mathcal{N}$ serves as a natural constraint so that all nontrivial critical points in $H$ are identical to critical points restricted on the manifold \cite{CLXY2025SISC}. Then the ground state solution $\mathbf{u}_g$ can be characterized by the minimizer of $E$ on $\mathcal{N}$. This inspires us to transform the computation of unstable ground state solutions of coupled system~\eqref{eq:model} into the following minimization problem on the manifold: 
\begin{align}\label{eq:min_E_on_N}
\mbox{Find $\mathbf{u}_g\in\mathcal{N}$ \; s.t.}\quad 
E(\mathbf{u}_g)=\min_{ \mathbf{u} \in \mathcal{N}} E(\mathbf{u}).
\end{align}
The variational formulation \eqref{eq:min_E_on_N} offers a numerically implementable pathway for stably computing the ground state solution.

To address the minimization problem of a generic functional on Nehari manifold, the NMOM \cite{CLXY2025SISC} introduces a Riemannian optimization framework on $\mathcal{N}$. This framework  consists of two main ingredients: one is a retraction mapping, which ensures that all iterates remain on the Nehari manifold; and the other is a tangential search direction (typically a descent direction) combined with a suitable step-size rule to decrease the functional. Under a nonmonotone step-size strategy, global convergence of the Nehari manifold optimization algorithm  has been established in \cite{CLXY2025SISC}. Particularly, the NMOM introduces the Riemannian steepest-descent algorithm on the Nehari manifold (RSD-$\mathcal{N}$)  to successfully compute the ground state solutions of a class of semilinear elliptic equations. However, a steepest-descent-type optimization algorithm often exhibits a slow convergence rate, especially when the target solution is ill-conditioned or structurally delicate. For instance, as the coupling matrix $\mathbf{g}$ varies, the ground state solution of the coupled system~\eqref{eq:model} may change from a fully nontrivial state (i.e., all components are nonzero) to a semi-trivial state with some components vanishing. Near such transitions, the solution becomes highly sensitive, and the number of iterations required for convergence increases significantly. 
Moreover, for the coupled system~\eqref{eq:model}, each iteration of the Nehari manifold optimization  algorithm requires solving $2m$ Poisson equations for computing the Riemannian gradient direction, which is more expensive than in the scalar case. As a result, the overall computational cost can become prohibitive, underscoring the need for more efficient algorithms within the NMOM framework.

In the unconstrained setting, the Nesterov accelerated gradient (NAG) method is a widely used acceleration scheme. Its core idea is to perform a linear extrapolation using the two most recent iterates, followed by a gradient-descent update at the extrapolated point. With only a modest additional cost per step, this extrapolation yields significantly faster convergence than standard steepest-descent one~\cite{NAG}. The continuous time limit of the NAG method can be interpreted as a second-order gradient flow, and its acceleration properties have been rigorously justified~\cite{SuWeijie-2016}. NAG-type ideas have also been successfully applied to compute ground states of rotating Bose–Einstein condensates, demonstrating high efficiency~\cite{CHEN_JCP2023}. Inspired by these successes, we aim to incorporate an extrapolation mechanism on the Nehari manifold and combine it with the RSD-$\mathcal{N}$ of the NMOM to obtain a Riemannian accelerated gradient algorithm on the Nehari manifold (RAG-$\mathcal{N}$). In the Riemannian optimization, the NAG method has attracted considerable interest, with extrapolation typically carried out along the geodesic connecting the two latest iterates via the exponential mapping~\cite{Liu2017RAG,Zhang_pmlr2018}. However, the exponential mapping on the (infinite-dimensional) Nehari manifold $\mathcal{N}$ may not be available, and even when it exists, computing it requires solving a geodesic equation, which is computationally demanding. Therefore, directly transplanting the exponential-mapping-based NAG idea to the Nehari manifold is not practical. To generate the RAG-$\mathcal{N}$  is still full of challenges.

In this paper, we construct a curve connecting the previous two iterates on $\mathcal{N}$ by combining the linear extrapolation with a suitable pullback operation onto the Nehari manifold $\mathcal{N}$, which serves as a simple surrogate for the computationally expensive geodesic between them. The extrapolated point is then chosen along this curve. It leads to a new  nonlinear extrapolation. Compared to exponential-map-based extrapolation, the proposed nonlinear extrapolation has an explicit form and only requires the evaluation of two integrals, making it easy to implement. After obtaining the extrapolated point, we apply the RSD step of the NMOM starting from this point, thus yielding the RAG-$\mathcal{N}$ algorithm. To further enhance robustness, we combine this accelerated scheme with a nonmonotone step-size search strategy.

The remainder of this paper is organized as follows. Section~\ref{sec:NMOM} presents the local variational characterization of the ground state solutions on the Nehari manifold $\mathcal{N}$ for the coupled system~\eqref{eq:model}. Based on this, Section~\ref{sec:Algorithms} proposes the efficient RAG-$\mathcal{N}$ algorithms. In Section~\ref{Sec:numerical result}, we report detailed numerical results that illustrate the feasibility and efficiency of the RAG-$\mathcal{N}$ algorithms for the coupled system~\eqref{eq:model} with two, three and four components, and investigate their behavior under various settings, including Gaussian-type external potentials,  singular diffusion coefficients and different coupling coefficients.

\vspace{-1ex}
\paragraph{\bf Notations} 
For $ \mathbf{u} = (u_1,u_2,\ldots,u_m)\in H$ and $\mathbf{v} = (v_1,v_2,\ldots,v_m) \in H$, define the inner product $(\cdot,\cdot)_H$ and the induced norm $\|\cdot\|_H$ as
 \begin{align}\label{eq:in_p}
 (\mathbf{u},\mathbf{v})_H = \int_{\Omega} \sum_{i=1}^m \varepsilon_i \nabla  u_i(\mathbf{x}) \cdot \nabla v_i(\mathbf{x}) d\mathbf{x}, \quad \|\mathbf{u}\|_H = \sqrt{(\mathbf{u},\mathbf{u})_H}.
 \end{align}
By applying the Poincar\'{e} inequality,  $H$ equipped with the above inner product $(\cdot,\cdot)_H$ is indeed a Hilbert space. The $H$-gradient of the functional $E$ at $\mathbf{u}\in H$, denoted as $\nabla E(\mathbf{u})\in H$, is defined as the Riesz representer of $E'(\mathbf{u})$, i.e.,
 \begin{align}\label{def:Hgrad}
(\nabla E(\mathbf{u}),\mathbf{v})_H = \langle E'(\mathbf{u}),\mathbf{v} \rangle \quad \forall \;\mathbf{v} \in H, 
 \end{align}
For convenience, we use the notation $\|\cdot\|_{\infty}$ to denote the $L^\infty$-norm for functions and the maximum norm for vectors/matrices.

\section{Preliminaries}\label{sec:NMOM}
In this section, some special properties of the variational functional $E$~\eqref{eq:energy} and the geometric settings of the  associated Nehari manifold $\mathcal{N}$ are analyzed. Then the local minimization variational characterization for the unstable ground state solutions of the coupled system~\eqref{eq:model} is established on the Nehari manifold $\mathcal{N}$.

\subsection{Variational settings}
For convenience, we rewrite $E(\mathbf{u}) = \frac{1}{2}K(\mathbf{u}) - \frac{1}{4}I(\mathbf{u})$ with functionals $K(\mathbf{u})$ and $I(\mathbf{u})$ defined as
\[K(\mathbf{u}) = \int_{\Omega} \sum_{i=1}^{m}\left(\varepsilon_i|\nabla u_i|^2 + a_i(\mathbf{x})u_i^2\right) d \mathbf{x}\quad \text{ and }\quad I(\mathbf{u}) =  \int_{\Omega}\sum_{i,j=1}^{m} g_{ij}u_i^2u_j^2  d\mathbf{x}.\] 
Utilizing the Poincar\'e inequality, the H\"older inequality and the Sobolev embedding $H_0^1(\Omega)\hookrightarrow L^{4}(\Omega)$, one can conclude that there exist positive constants $c_1,c_2,c_3>0$ s.t.
   \begin{align}\label{eq:sob_1}
    c_1 \|\mathbf{u}\|_H^2 \leq K(\mathbf{u}) \leq c_2 \|\mathbf{u}\|_H^2, \quad 
    I(\mathbf{u}) \leq c_3\|\mathbf{u}\|_H^4 \quad \forall\, \mathbf{u} \in H.
    \end{align}
In addition, $E\in C^2(H,\mathbb{R})$, and for all  $\mathbf{v}, \mathbf{w} \in H$, its first- and second-order Fr\'echet derivatives at $\mathbf{u}$ satisfy 
 \begin{align}
 \label{eq:E'uv}
     \langle E'(\mathbf{u}), \mathbf{v} \rangle &= \int_{\Omega} \sum_{i=1}^m \left(\varepsilon_i \nabla u_i \cdot \nabla v_i + a_i u_iv_i \right) d\mathbf{x} - \int_{\Omega} \sum_{i,j=1}^m g_{ij}u_j^2u_iv_i  d\mathbf{x}, \\
 \label{eq:E''uwv}
     \langle E''(\mathbf{u})\mathbf{w}, \mathbf{v}\rangle &= \int_{\Omega} \sum_{i=1}^m \left(\varepsilon_i\nabla w_i \cdot \nabla v_i + a_i w_i v_i \right) d\mathbf{x} - \int_{\Omega}\sum_{i,j=1}^m g_{ij} \left(u_j^2w_iv_i  + 2u_jw_ju_iv_i  \right) d\mathbf{x}.
 \end{align}
Let us introduce a functional $G:H\to \mathbb{R}$ as
\begin{align}\label{eq:def_G}
G(\mathbf{u}) := \langle E'(\mathbf{u}),\mathbf{u}\rangle= K(\mathbf{u}) - I(\mathbf{u}).
\end{align}
The Nehari manifold $\mathcal{N}$ \eqref{eq:Nehari manifold-1} can be expressed as 
\(\mathcal{N} = \left\{ \mathbf{u} \in H\backslash \{0\}: G(\mathbf{u}) = 0\right\}. \)

Clearly, when $\mathbf{u} \in \mathcal{N}$, one has $K(\mathbf{u}) = I(\mathbf{u}) = K^2(\mathbf{u})/I(\mathbf{u}) \geq c_1^2/c_3>0$ and $E(\mathbf{u}) = \frac{1}{4}K(\mathbf{u})=\frac{1}{4}I(\mathbf{u})$. We have the following basic properties for the variational functional $E$. 
\begin{lemma}\label{lem:basic_properties}
   Under Assumption~\ref{ass:A}, the followings hold:
   \begin{itemize}
        \item[(i)] For each $\mathbf{v} = (v_1,v_2,\ldots,v_m) \in H\backslash \{0\}$, the function $\Phi(s) = E(s\mathbf{v})$ w.r.t. $s\in (0,\infty)$ owns a unique critical point $s_{\mathbf{v}}>0$, and $s_{\mathbf{v}} {\mathbf{v}} \in \mathcal{N}$. Specifically, 
         \begin{align}\label{eq:rho}
            s_{\mathbf{v}} = \left(\frac{K(\mathbf{v})}{I(\mathbf{v})} \right) ^{{1}/{2}}  = \left(\frac{\sum\limits_{i=1}^{m}\int_{\Omega}  \left(\varepsilon_i |\nabla
	        v_i|^2+ a_iv_i^2 \right)d \mathbf{x}}{\sum\limits_{i,j=1}^m \int_{\Omega} g_{ij}v_{i}^2v_j^2 d \mathbf{x}} \right)^{{1}/{2}}.
         \end{align}
        \item[(ii)] For all $\mathbf{u} \in \mathcal{N}$, $\langle E''(\mathbf{u})\mathbf{u},\mathbf{u} \rangle <0$. 
        \item[(iii)] There exists a  positive constant $\sigma_0 >0$ s.t.
    \begin{align} 
        \|\mathbf{u}\|_H \geq \sigma_0 \quad\; \forall\, \mathbf{u} \in \mathcal{N}. \label{eq:sob_11}
    \end{align}
    \end{itemize}
\end{lemma}
\begin{proof}
For each $\mathbf{v}\in H\backslash \{0\}$, one has $\Phi(s) =E(s\mathbf{v})=\frac{s^2}{2}K(\mathbf{v})-\frac{s^4}{4}I(\mathbf{v})$ with $K(\mathbf{v})>0$ and $I(\mathbf{v})>0$. $(i)$ is straightforward. Then, for all $\mathbf{u} \in \mathcal{N}$, by applying \eqref{eq:E''uwv}, a direct computation gives that 
\[
\langle E''(\mathbf{u})\mathbf{u},\mathbf{u} \rangle =- 2\int_{\Omega}\sum_{i,j=1}^{m} g_{ij}u_i^2u_j^2  d\mathbf{x}  < 0,
\]
which leads to $(ii)$. Moreover, by applying \eqref{eq:sob_1} and the fact $K(\mathbf{u}) = I(\mathbf{u})$ for $\mathbf{u} \in  \mathcal{N}$, one has $\| \mathbf{u}\|_H \geq \sqrt{{c_1}/{c_3 }}=: \sigma_0$. Thus, $(iii)$ is obtained. The proof is completed.
\end{proof}

Intuitively, $(i)$ and  $(ii)$ in Lemma~\ref{lem:basic_properties} show that $s_{\mathbf{v}}$ is the global maximizer of $\Phi(s) = E(s\mathbf{v})$ in $(0,\infty)$, and $(iii)$ states that  $\mathcal{N}$ is bounded away from 0 and therefore a closed subset of $H$. Actually, Lemma~\ref{lem:basic_properties} verifies the basic assumptions $(A1)$-$(A3)$ in \cite{CLXY2025SISC} for the variational functional \eqref{eq:energy} of the coupled system~\eqref{eq:model}.

\subsection{Local minimization characterization on \texorpdfstring{$\mathcal{N}$}{N}} 
In the following, we consider the local minimization characterization of the ground state solution $\mathbf{u}_g$ on the Nehari manifold $\mathcal{N}$. Recalling the definition of $G$ in \eqref{eq:def_G}, $(ii)$ in Lemma~\ref{lem:basic_properties} illustrates that $G(\mathbf{u})=0$ defines a regular (nondegenerate) constraint, i.e., $G'(\mathbf{u})\neq0$ for all $\mathbf{u}\in\mathcal{N}$, which is a direct conclusion from the fact $\langle G'(\mathbf{u}),\mathbf{u}\rangle=\langle E''(\mathbf{u})\mathbf{u},\mathbf{u}\rangle<0$ on $\mathcal{N}$. This essentially ensures the smoothness of the Nehari manifold $\mathcal{N}$. In fact, based on Lemma~\ref{lem:basic_properties}, one can apply \cite[Lemma~2.1]{CLXY2025SISC} to confirm that $\mathcal{N}$ is a closed $C^1$ submanifold\footnote{Recall that a closed subset $ \mathcal{M} \subset H $ is closed $C^m$ submanifold ($m\geq1$) of $H$ if the set of charts in $ \mathcal{M}$ which are restrictions of charts for $H$ form an atlas for $\mathcal{M}$ of class $C^m$ \cite{1995differential,Palais1963}.} of $H$.

Then, the \textbf{tangent space} of $\mathcal{N}$ at $\mathbf{u} = (u_1,u_2,\ldots,u_m) \in \mathcal{N}$ is given by  
  \begin{align}
  T_{\mathbf{u}}  \mathcal{N} & = \{\boldsymbol{\xi}= (\xi_1,\xi_2,\ldots,\xi_m) \in H: \langle G'(\mathbf{u}), \boldsymbol{\xi} \rangle = 0\} \notag \\  
  & = \bigg\{\boldsymbol{\xi} \in H: \int_{\Omega }\sum_{i=1}^m \left(\varepsilon_i\nabla u_i \cdot \nabla \xi_i + a_iu_i \xi_i \right) d\mathbf{x} = 2\int_{\Omega } \sum_{i,j=1}^m  g_{ij}u_iu_j^2 \xi_i d\mathbf{x} \bigg\}.
  \end{align}
Clearly, $\mathcal{N}$ owns a Riemannian geometry by naturally inheriting the Riemannian metric $(\cdot,\cdot)_H$ of the Hilbert space $H$. The {\bf Riemannian gradient} of $E$ at $\mathbf{u}\in\mathcal{N}$, denoted by $\nabla_{\mathcal{N}} E(\mathbf{u})$, is defined as the unique element in the tangent space $T_{\mathbf{u}} \mathcal{N}$, s.t. 
 \begin{align}\label{def:R-grad}
 (\nabla_{\mathcal{N}} E(\mathbf{u}),\bm{\xi})_H=\langle E'(\mathbf{u}),\bm{\xi}\rangle\quad \forall\, \bm{\xi}\in T_{\mathbf{u}} \mathcal{N}. 
\end{align}

By comparing \eqref{def:R-grad} with the definition of the $H$-gradient $\nabla E(\mathbf{u})$ in \eqref{def:Hgrad}, one can conclude that $\nabla_{\mathcal{N}} E(\mathbf{u})$ is the $H$-orthogonal projection of $\nabla E(\mathbf{u})$ onto $T_{\mathbf{u}} \mathcal{N}$, i.e., 
\begin{align}\label{eq:Rgrad}
\nabla_{\mathcal{N}}E(\mathbf{u}) = \nabla E(\mathbf{u}) - \frac{(\nabla E(\mathbf{u}),\nabla G(\mathbf{u}))_H}{\|\nabla G(\mathbf{u})\|_H^2} \nabla G(\mathbf{u}).
\end{align}
Then we can establish the following estimates for the norm of the Riemannian gradient and $H$-gradient of the functional $E$ in \eqref{eq:energy} with the proof illustrated in Appendix~\ref{sec:proof-lem:Hgrad-vs-Rgrad-norm}.
\begin{lemma}\label{lem:Hgrad-vs-Rgrad-norm}
Under Assumption~\ref{ass:A}, there exists a constant $C>0$, s.t.
    \begin{align}\label{eq:Hgrad-vs-Rgrad-norm}
     \left\|\nabla_{\mathcal{N}}E(\mathbf{u}) \right\|_H 
     \leq \left\|\nabla E(\mathbf{u}) \right\|_H 
     \leq C\left\|\nabla_{\mathcal{N}}E(\mathbf{u}) \right\|_H\left(1+E(\mathbf{u})\right)\quad \forall\, \mathbf{u}\in\mathcal{N}.
    \end{align}
\end{lemma}

Note that all nontrivial critical points of $E$ lie in $\mathcal{N}$. Moreover, Lemma~\ref{lem:Hgrad-vs-Rgrad-norm} implies the equivalence between nontrivial critical points of $E$ in $H$ (where $\nabla E(\mathbf{u}) =0$) and the critical points of $E$ in $\mathcal{N}$ (where $\nabla_{\mathcal{N}} E(\mathbf{u}) =0$). In other words, the Nehari manifold $\mathcal{N}$ is exactly a {\em natural constraint} for nontrivial critical points of $E$.

Further, we can establish the variational characterization of the 1-saddles of $E$ in $H$ with $\mathrm{MI} = 1$ via the local minimization on $\mathcal{N}$. By virtue of Lemmas~\ref{lem:basic_properties}-\ref{lem:Hgrad-vs-Rgrad-norm}, the following proposition is a direct conclusion of \cite[Theorem~2.4]{CLXY2025SISC}.
\begin{proposition}\label{thm:saddle-minN}
For the functional $E$ defined in~\eqref{eq:energy}, under Assumption~\ref{ass:A}, all local minimizers of $E$ in $\mathcal{N}$ are 1-saddles of $E$ in $H$, and all nontrivial and nondegenerate 1-saddles of $E$ in $H$ are strict local minimizers of $E$ in $\mathcal{N}$.
\end{proposition}

Recall that the ground state solutions of the coupled system~\eqref{eq:model} are those  with the least variational functional value among all nontrivial solutions. Lemma~\ref{lem:Hgrad-vs-Rgrad-norm} and Proposition~\ref{thm:saddle-minN} show that these ground state solutions are 1-saddles of $E$ in $H$, and can be obtained by solving the minimization problem on $\mathcal{N}$, as formulated in \eqref{eq:min_E_on_N}. Directly minimizing $E$ in $H$ to compute the ground state solution $\mathbf{u}_g$ (1-saddle of $E$ in $H$) is impractical, while the constrained minimization problem \eqref{eq:min_E_on_N} provides a numerically stable computational model. In addition, we remark that the existence of the ground state solution $\mathbf{u}_g$ of the coupled system~\eqref{eq:model}, which is essentially known in literature \cite{ANTONIO_NLSE_jlms2007}, can be obtained by applying \cite[Theorem~2.7]{CLXY2025SISC} with the following Palais-Smale (PS) compactness on the Nehari manifold $\mathcal{N}$. 
\begin{lemma}\label{lem:PS}
Under Assumption~\ref{ass:A}, the functional $E$ defined in~\eqref{eq:energy} satisfies the {\rm PS} condition on the Nehari manifold $\mathcal{N}$, i.e., every sequence $\{\mathbf{u}_n\}\subset\mathcal{N}$ s.t. $\{E(\mathbf{u}_n)\}$ is bounded and $\|\nabla_{\mathcal{N}} E(\mathbf{u}_n)\|_H \to 0$ (as $n\to \infty$) has a convergent subsequence in $H$.
\end{lemma}
\begin{proof}
Let $\{\mathbf{u}_n\}$ be a sequence in $\mathcal{N}$ s.t. $\{E(\mathbf{u}_n)\}$ is bounded and $\|\nabla_{\mathcal{N}} E(\mathbf{u}_n)\|_H \to 0$ as $n\to \infty$. 
As a straightforward conclusion of Lemma~\ref{lem:Hgrad-vs-Rgrad-norm}, one has $\|\nabla  E(\mathbf{u}_n)\|_H \to 0$ as $n\to \infty$. Then, $\{\mathbf{u}_n\}$ is a PS sequence in $H$. By the standard result in critical point theory \cite{Rabinowitz1986}, the functional $E$ defined in~\eqref{eq:energy} satisfies the PS condition in $H$, i.e., every PS sequence in $H$ has a convergent subsequence in $H$. The proof is completed.
\end{proof}

Now the mathematical justifications of the NMOM proposed in \cite{CLXY2025SISC} are verified for the coupled system~\eqref{eq:model} to address the minimization problem~\eqref{eq:min_E_on_N}.

\section{Accelerated Nehari manifold optimization algorithms }\label{sec:Algorithms}
In this section, we firstly introduce the RSD-$\mathcal{N}$ algorithm  of NMOM for solving the local minimization problem~\eqref{eq:min_E_on_N} on $\mathcal{N}$.  To improve the efficiency, we propose the RAG-$\mathcal{N}$ algorithm, and further combine the RAG-$\mathcal{N}$ algorithm  with nonmonotone search rule to generate a more stable one. 

\subsection{Nehari retraction and RSD-\texorpdfstring{$\mathcal{N}$}{N} algorithm}
In contrast to unconstrained optimization, where iterates are updated linearly, the Riemannian optimization requires mechanisms to ensure that each iterate remains on the manifold. 
While adopting the exponential mapping (which moves along exact geodesics) is theoretically natural, it is often computationally impractical; the so-called retractions offer a computationally efficient alternative in practice \cite{Absil2008,Boumal2023}. 
Following \cite{CLXY2025SISC}, a  retraction $\mathcal{R}$ is introduced on the Nehari manifold $\mathcal{N}$, which is a $C^1$ mapping from the tangent bundle $T\mathcal{N} = \bigcup_{\mathbf{u} \in \mathcal{N}} \{(\mathbf{u},\boldsymbol{\xi}):\boldsymbol{\xi} \in T_{\mathbf{u}} \mathcal{N}\}$ to $\mathcal{N}$ such that its restriction on $T_{\mathbf{u}} \mathcal{N}$, denoted as $\mathcal{R}_{\mathbf{u}}(\cdot) = \mathcal{R}(\mathbf{u},\cdot)$, satisfies
\begin{enumerate}
    \item[(i)] $\mathcal{R}_{\mathbf{u}}(0) = \mathbf{u}$; 
    \item[(ii)] $\left.\frac{d}{dt}\mathcal{R}_{\mathbf{u}}(t\boldsymbol{\xi}) \right|_{t=0} = \boldsymbol{\xi}$ for all $(\mathbf{u},\boldsymbol{\xi}) \in T\mathcal{N}$. 
\end{enumerate}

The NMOM uses the following update scheme \cite{CLXY2025SISC}
\begin{align}\label{eq:upd_form}
\mathbf{u}_{n+1} = \mathcal{R}_{\mathbf{u}_n}(\alpha_n\boldsymbol{\eta}_n),\quad n=0,1,\ldots,
\end{align}
where $\boldsymbol{\eta}_n \in T_{\mathbf{u}_n}  \mathcal{N}$ is the tangential search direction and $\alpha_n>0$ is the step-size. We remark that there are various choices for the search direction $\boldsymbol{\eta}_n$, the step-size $\alpha_n$ and the retraction $\mathcal{R}$ in the NMOM framework.

A simple idea under the update scheme~\eqref{eq:upd_form} is the Riemannian steepest-descent algorithm on the Nehari manifold  (RSD-$\mathcal{N}$) \cite{CLXY2025SISC}, which adopts the Riemannian steepest-descent direction $\boldsymbol{\eta}_n=- \nabla_{\mathcal{N}}E(\mathbf{u}_n)$ defined in \eqref{eq:Rgrad}, i.e.,
\begin{align}\label{eq:RSD}
    \mathbf{u}_{n+1} = \mathcal{R}_{\mathbf{u}_n}\left(- \alpha_n \nabla_{\mathcal{N}}E(\mathbf{u}_n)\right),\quad n=0,1,\ldots,
\end{align} 
where we use a fixed step-size $\alpha_n=\alpha>0$ for simplicity.

An explicit Nehari retraction and the Riemannian steepest-descent direction are computed as below. 
\begin{itemize}
\item {\bf Nehari retraction.}
By Lemma~\ref{lem:basic_properties}, the mapping $\rho: \mathbf{v}\mapsto s_{\mathbf{v}}$ defined as
\begin{align}\label{def:rho}
    \rho(\mathbf{v}) =\left( \frac{ K(\mathbf{v})}{I(\mathbf{v})} \right)^{{1}/{2}} = \left(\frac{\sum\limits_{i=1}^{m}\int_{\Omega}  \left(\varepsilon_i |\nabla v_i|^2+ a_iv_i^2 \right)d \mathbf{x}}{\sum\limits_{i,j=1}^m \int_{\Omega} g_{ij}v_{i}^2v_j^2d \mathbf{x}} \right)^{{1}/{2}} \quad \forall\, \mathbf{v}  = (v_1,\ldots,v_m) \in H \backslash \{0\}, 
\end{align}
is of class $C^1$ from $H\backslash \{0\}$ to $\mathbb{R}^+$, and $\rho(\mathbf{v})\mathbf{v} \in \mathcal{N}$ for every $\mathbf{v}\in H \backslash \{0\}$. As in \cite{CLXY2025SISC}, based on this property, the Nehari retraction for the coupled system~\eqref{eq:model} can be given explicitly as
\begin{align}\label{eq:Nehari retraction}
    \mathcal{R}_{\mathbf{u}}(\boldsymbol{\xi}) := 
    \rho\!\left(\mathbf{u} +\boldsymbol{\xi} \right) \left(\mathbf{u} +\boldsymbol{\xi} \right)
    =\left( \frac{ K(\mathbf{u} +\boldsymbol{\xi})}{I(\mathbf{u} +\boldsymbol{\xi})} \right)^{{1}/{2}}\! \left(\mathbf{u} +\boldsymbol{\xi} \right) \quad\; \forall\, (\mathbf{u}, \boldsymbol{\xi} ) \in T\mathcal{N}. 
\end{align}
\item {\bf Riemannian steepest-descent direction (RSD).}
The RSD of $E$ on the Nehari manifold $\mathcal{N}$ at $\mathbf{u}$ is given by $\boldsymbol{\eta}=- \nabla_{\mathcal{N}}E(\mathbf{u})$. 
From \eqref{eq:Rgrad}, the Riemannian gradient $\nabla_{\mathcal{N}} E(\mathbf{u})$ is determined by the $H$-gradients $\boldsymbol{\psi}=(\psi_1,\psi_2,\ldots,\psi_m):=\nabla E(\mathbf{u})$ and $\boldsymbol{\varphi} = (\varphi_1,\varphi_2,\ldots,\varphi_m):=\nabla G(\mathbf{u})$, yielding that
\begin{align}\label{eq:eta-SD}
\boldsymbol{\eta}=- \nabla_{\mathcal{N}}E(\mathbf{u})
=- \boldsymbol{\psi}+\frac{(\boldsymbol{\psi},\boldsymbol{\varphi})_H}{\|\boldsymbol{\varphi}\|_H^2}\boldsymbol{\varphi},
\end{align}
where $\boldsymbol{\psi}$ and $\boldsymbol{\varphi}$ are the solutions to the following $2m$ independent Poisson equations: 
\begin{align}\label{eq: R_grad_N}
\begin{cases}
- \varepsilon_i\Delta  \psi_i = -\varepsilon_i \Delta u_i + a_iu_i - b_i(\mathbf{u})u_i~~~~\quad\, \mbox{in }\Omega, & \psi_i|_{\partial\Omega}=0,\\
- \varepsilon_i \Delta \varphi_i = -2\varepsilon_i \Delta u_i + 2a_iu_i - 4b_i(\mathbf{u})u_i\quad \mbox{in }\Omega, &\varphi_i|_{\partial\Omega}=0,\\
 \end{cases}\quad i=1,2,\ldots,m.
\end{align} 
Here, $b_i(\mathbf{u}):=\sum_{j=1}^mg_{ij}u_j^2$. Clearly, when $a_1=\ldots=a_m=0$, there holds $\bm{\varphi}=4\bm{\psi}-2\mathbf{u}$, and the computation reduces to solving $m$ Poisson equations.
\end{itemize}

In the following, we develop a class of RAG-$\mathcal{N}$ algorithms for accelerating the computation of ground state solutions of the coupled system~\eqref{eq:model}.

\subsection{RAG-\texorpdfstring{$\mathcal{N}$}{N} algorithm}
The basic idea is inspired by the NAG method in unconstrained optimization problem $\min_{\mathbf{x}} f(\mathbf{x})$  in $\mathbb{R}^N$~\cite{NAG}, which is to firstly perform a linear extrapolation $\mathbf{y}_n = \mathbf{x}_{n} + t_n(\mathbf{x}_n - \mathbf{x}_{n-1})$ regarding the previous two iteration points $\mathbf{x}_{n}$ and $\mathbf{x}_{n-1}$ in $\mathbb{R}^N$ with a special parameter $t_n \in [0,1]$, and then update the $\mathbf{x}_{n+1}$ by the steepest-descent at $\mathbf{y}_n$, i.e. $\mathbf{x}_{n+1} = \mathbf{y}_n - \alpha \nabla f(\mathbf{y}_n)$. The term $t_n(\mathbf{x}_n - \mathbf{x}_{n-1})$ in the linear extrapolation is called a momentum and $t_n$ is the momentum coefficient. Compared to the traditional steepest-descent method, the linear extrapolation with a special parameter involved in the NAG method serves as the main mechanism for the acceleration. However, such a linear extrapolation cannot be directly applied to the nonlinear manifold such as the Nehari manifold. In the setting of the Riemannian optimization, the idea of extrapolation is generalized by other nonlinear operators related to the exponential mapping~\cite{Liu2017RAG,Zhang_pmlr2018}. However, for the coupled system~\eqref{eq:model}, the exponential mapping on the $C^1$ submanifold $\mathcal{N}$ may not exist. Even if it exists, it lacks explicit expression and its computation requires solving a geodesic equation, which is nonlinear and hard to implement. It is necessary to find an easy analogy to the linear extrapolation on the nonlinear submanifold $\mathcal{N}$.

Let $\mathbf{u}_{n-1}$ and $\mathbf{u}_{n}$ be two previous iteration points on the Nehari manifold $\mathcal{N}$, and define the linear extrapolation $\hat{\mathbf{w}}_n := \mathbf{u}_n +t_n(\mathbf{u}_n - \mathbf{u}_{n-1})$. Although $\hat{\mathbf{w}}_n\notin\mathcal{N}$ in general, Lemma~\ref{lem:linear exploa} below ensures that $\hat{\mathbf{w}}_n\neq0$ whenever $t_n\neq-1/2$ or $\mathbf{u}_{n}\neq-\mathbf{u}_{n-1}$, and the mapping $\rho$ defined in \eqref{def:rho} can be used to pull it back into $\mathcal{N}$ by $ \mathbf{w}_n = \rho(\hat{\mathbf{w}}_n)\hat{\mathbf{w}}_n$. Moreover, when $\mathbf{u}_{n}\neq-\mathbf{u}_{n-1}$, 
\begin{align}\label{eq:curve-extrap}
    \gamma(t) := \rho\left(\mathbf{u}_n  +t(\mathbf{u}_{n} - \mathbf{u}_{n-1})\right)\left(\mathbf{u}_n  +t(\mathbf{u}_{n} - \mathbf{u}_{n-1})\right),\quad t \in \mathbb{R},
\end{align} 
is a $C^1$ curve on $\mathcal{N}$ connecting $\gamma(-1) = \mathbf{u}_{n-1}$ and $\gamma(0) = \mathbf{u}_{n}$. The nonlinear extrapolation on $\mathcal{N}$ is implemented along this curve by choosing a specific $t = t_n$. The corresponding Riemannian accelerated gradient method on $\mathcal{N}$ (RAG-$\mathcal{N}$) is proposed by implementing the iterative scheme
\begin{align}\label{eq:RNAG}
\begin{cases}
\hat{\mathbf{w}}_n = \mathbf{u}_n +t_n(\mathbf{u}_n - \mathbf{u}_{n-1}), \\
\mathbf{w}_n = \rho(\hat{\mathbf{w}}_n) \hat{\mathbf{w}}_n, \\
\mathbf{u}_{n+1} = \mathcal{R}_{\mathbf{w}_{n}}(- \alpha \nabla_{\mathcal{N}} E(\mathbf{w}_{n})),\\
\end{cases} \quad n=1,2,\ldots,
\end{align}
with the given initial data $\mathbf{u}_1 = \mathbf{u}_0\in\mathcal{N}$. Actually, the combination of the first and second steps in~\eqref{eq:RNAG} performs exactly a nonlinear extrapolation along the curve \eqref{eq:curve-extrap} at $t=t_n$.

The geometric interpretation of the RAG-$\mathcal{N}$ method for the iterative scheme \eqref{eq:RNAG} is depicted in Figure~\ref{fig:RAG-1}. When $\mathbf{u}_{n}$ is sufficiently close to $\mathbf{u}_{n-1}$ (typically, when $\{\mathbf{u}_n\}$ is close to convergence), the following Lemma~\ref{lem:linear exploa} explains that the linear extrapolation $\hat{\mathbf{w}}_n$ is sufficiently close to  $\mathcal{N}$, and the pullback from $\hat{\mathbf{w}}_n$ to $ \mathbf{w}_n$ is only a small correction with magnitude $o(\|\mathbf{u}_n - \mathbf{u}_{n-1}\|_H)$ (as shown in Figure~\ref{fig:RAG-1}). Consequently, $\mathbf{w}_n = \rho(\hat{\mathbf{w}}_n)\hat{\mathbf{w}}_n$ specifies a simple extrapolation procedure along $\mathcal{N}$ with $\hat{\mathbf{w}}_n$ as its dominant component.

\begin{lemma}\label{lem:linear exploa}
Under Assumption~\ref{ass:A}, let $\mathbf{u},\bar{\mathbf{u}}\in \mathcal{N}$ and $\hat{\mathbf{w}}=\mathbf{u}+t(\mathbf{u}-\bar{\mathbf{u}})$ with $t\in\mathbb{R}$. Then \\ 
(i) when $t\neq -1/2$ or $\mathbf{u}\neq-\bar{\mathbf{u}}$, we have $\hat{\mathbf{w}}\neq0$ so that $\mathbf{w}:=\rho(\hat{\mathbf{w}})\hat{\mathbf{w}}\in\mathcal{N}$; 
and \\
(ii) for every fixed $\mathbf{u}$, when $\|\bar{\mathbf{u}} - \mathbf{u}\|_H $ is sufficiently small and $t$ lies in a bounded interval, 
\begin{align*}
\mathrm{dist}(\hat{\mathbf{w}},\mathcal{N}) \leq \|\mathbf{w}-\hat{\mathbf{w}}\|_H  =  o\left(\|\mathbf{u} -\bar{\mathbf{u}}\|_H\right),
\end{align*}
where $\mathrm{dist}(\hat{\mathbf{w}},\mathcal{N}):=\inf_{\tilde{\mathbf{w}}\in\mathcal{N}}\| \hat{\mathbf{w}} - \tilde{\mathbf{w}}\|_H$.
\end{lemma}
\begin{proof}
Supposing $\hat{\mathbf{w}}=0$, one gets $(t+1)\mathbf{u}=t\bar{\mathbf{u}}$, which means clearly that $t\notin\{-1,0\}$. Since $\rho(\mathbf{u})=\rho(\bar{\mathbf{u}})=1$, by \eqref{def:rho}, one has $|t|=|t|\rho(\mathbf{u})=\rho\left(\frac1t\mathbf{u}\right)=\rho\left(\frac{1}{t+1}\bar{\mathbf{u}}\right)=|t+1|\rho(\bar{\mathbf{u}})=|t+1|$, yielding that $t=-1/2$ and $\mathbf{u}=-\bar{\mathbf{u}}$, a contradiction. $(i)$ is proved. In the following, we verify $(ii)$. Since $\mathbf{w}\in\mathcal{N}$, $\mathrm{dist}(\hat{\mathbf{w}},\mathcal{N}) \leq \|\mathbf{w} -\hat{\mathbf{w}}\|_H$. Moreover, the facts $\rho(\mathbf{u})=\rho(\bar{\mathbf{u}})=1$ imply that $0= \rho(\bar{\mathbf{u}})-\rho(\mathbf{u})= \langle\rho'(\mathbf{u}), \bar{\mathbf{u}}-\mathbf{u}\rangle + o\left(\|\mathbf{u}-\bar{\mathbf{u}}\|_H\right)$. It immediately leads to $\langle\rho'(\mathbf{u}), \mathbf{u}-\bar{\mathbf{u}}\rangle = o\left(\|\mathbf{u}-\bar{\mathbf{u}}\|_H\right)$. Then, noticing $\mathbf{w}=\rho(\hat{\mathbf{w}})\hat{\mathbf{w}}$, $\rho(\mathbf{u})=1$, and $\|\hat{\mathbf{w}}\|_H=\|\mathbf{u}+t(\mathbf{u}-\bar{\mathbf{u}})\|_H \leq \|\mathbf{u}\|_H+|t|\|\mathbf{u}-\bar{\mathbf{u}}\|_H$, one has
\begin{align*}
\|\mathbf{w}-\hat{\mathbf{w}}\|_H 
&= \|\hat{\mathbf{w}}\|_H \big|\rho(\hat{\mathbf{w}})-\rho(\mathbf{u})\big| \\
&= \|\hat{\mathbf{w}}\|_H \Big| \left\langle\rho'(\mathbf{u}), \hat{\mathbf{w}}-\mathbf{u}\right\rangle + o\left(\|\hat{\mathbf{w}}-\mathbf{u}\|_H\right) \Big| \\
&= |t|\|\hat{\mathbf{w}}\|_H \Big| \left\langle\rho'(\mathbf{u}), \mathbf{u}-\bar{\mathbf{u}}\right\rangle + o\left(\|\mathbf{u}-\bar{\mathbf{u}}\|_H\right)\Big|  \\
&=o\left(\|\mathbf{u}-\bar{\mathbf{u}}\|_H\right),
\end{align*}
yielding the conclusion in $(ii)$. The proof is finished.
\end{proof}

\begin{figure}[!t]
    \centering
    \includegraphics[width = 0.49\textwidth]{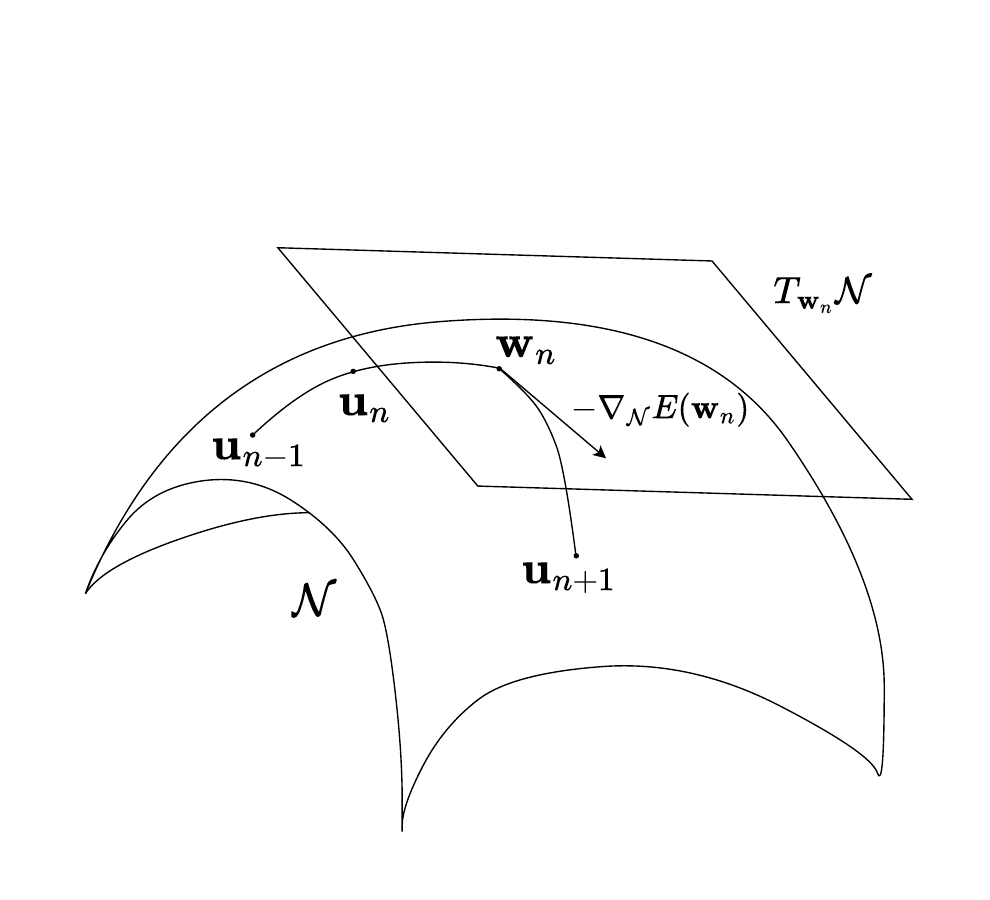}
     \includegraphics[width = 0.49\textwidth]{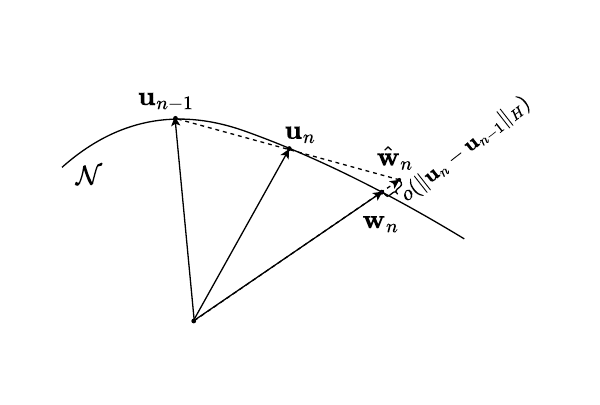}
    \caption{Illustration of the iterative scheme \eqref{eq:RNAG} (left) and the extrapolation (right).}
    \label{fig:RAG-1}
\end{figure}
The choice of the parameter $t_n$ is important for accelerating the convergence. Since the nonlinear extrapolation in~\eqref{eq:RNAG} closely resembles linear extrapolation, we may directly use the original scheme in NAG~\cite{NAG} to compute $t_n$. We take $t_n =(\theta_{n-1}-1)/\theta_n$ ($n\geq1$) with $\theta_{n}$ iteratively computed by
\begin{align}\label{eq:mom_coe}
    \theta_0 = 0,\quad 
        \theta_{n} = \frac{1+\sqrt{1+4\theta_{n-1}^2}}{2}, \quad n=1,2,\ldots.
\end{align}
Clearly, $\{\theta_n\}$ is strictly increasing and $t_1=-1$, $t_n\in[0,1)$ for $n\geq2$. Then, by Lemma~\ref{lem:linear exploa}, the nonlinear extrapolation in \eqref{eq:RNAG} remains valid even in the extreme case when $\mathbf{u}_{n}=-\mathbf{u}_{n-1}$. In addition, $\theta_n$ and $t_n$ have the asymptotic formulas $\theta_n\sim n/2$ and $t_n\sim(n-2)/(n+1)$, respectively, for large $n$.

By incorporating the expression for the Nehari retraction in \eqref{eq:Nehari retraction}, the mapping $\rho$ in \eqref{def:rho} and the Riemannian steepest-descent direction described in \eqref{eq:eta-SD}-\eqref{eq: R_grad_N}, the proposed RAG-$\mathcal{N}$ algorithm for the iterative scheme \eqref{eq:RNAG} is detailed in Algorithm~\ref{al:RNAG}. As will be demonstrated in subsequent numerical experiments, the RAG-$\mathcal{N}$ algorithm exhibits a clear acceleration for the original NMOM algorithm in terms of the number of iterations.

\begin{algorithm}[!t]
	\caption{RAG-$\mathcal{N}$ algorithm}	\label{al:RNAG}
	Initialize $\mathbf{u}_1 =\mathbf{u}_0 \in \mathcal{N}$, $\theta_0 = 0$ and a fixed step-size $\alpha>0$.  Set $n=1$.\\
	\While {stop criterion is not met} 
	{ 
    Update $\theta_n$ by~\eqref{eq:mom_coe} and set $t_n=(\theta_{n-1}-1)/\theta_n$. \\
    Compute $\hat{\mathbf{w}}_n = \mathbf{u}_n +t_n(\mathbf{u}_n - \mathbf{u}_{n-1})$ and $\mathbf{w}_n = \rho(\hat{\mathbf{w}}_n) \hat{\mathbf{w}}_n$ by employing \eqref{def:rho}.\\
    Solve \eqref{eq: R_grad_N} with $\mathbf{u}=\mathbf{w}_n$ and compute $\boldsymbol{\eta}_n=- \nabla_{\mathcal{N}}E(\mathbf{w}_n)$ according to \eqref{eq:eta-SD}. \\
    Compute $\mathbf{u}_{n+1}=\mathcal{R}_{\mathbf{w}_{n}}(\alpha\bm{\eta}_n)$ by employing \eqref{eq:Nehari retraction} and \eqref{def:rho}. \\
    $n: = n+1$.
	}
\end{algorithm}

\subsection{RAG-\texorpdfstring{$\mathcal{N}$}{N} algorithm with nonmonotone search (nmRAG-\texorpdfstring{$\mathcal{N}$}{N} algorithm)}\label{sec:Nom-AG}
The Algorithm~\ref{al:RNAG} is essentially a nonmonotone algorithm. When the step-size $\alpha$ is chosen too large, the nonmonotone behavior would become more severe and eventually lead to a divergence, as observed in Section \ref{Sec:numerical result}. In this subsection, we propose an improved RAG-$\mathcal{N}$ algorithm with better robustness.

The RAG-$\mathcal{N}$ algorithm in Algorithm \ref{al:RNAG} can be viewed as an extrapolated variant of the original RSD-$\mathcal{N}$ algorithm. From the perspective of the RSD, a common strategy to make the iterative process more effective and reliable is to incorporate a suitable step-size search rule. Since the RAG-$\mathcal{N}$ algorithm in Algorithm \ref{al:RNAG} produces sequences that are typically nonmonotone w.r.t. the variational functional $E$, a natural idea is to combine it with a certain nonmonotone step-size search, such that each iteration satisfies a flexible functional value-decrease condition while maintaining high efficiency. Inspired by~\cite{APG2015}, at each iteration we compute both the RAG-$\mathcal{N}$ update in \eqref{eq:RNAG} and the RSD update \eqref{eq:RSD} with a nonmonotone step-size search rule \cite{CLXY2025SISC}, and then select the one with the lower variational functional value as the next iterative point. The implementation for the RAG-$\mathcal{N}$ method with the nonmonotone step-size search (nmRAG-$\mathcal{N}$) is detailed in the following iterative scheme:
\begin{align} 
 \mathbf{u}_{n+1}  = 
    \begin{cases}
       \mathbf{z}_{n+1}, & \text{if } E(\mathbf{z}_{n+1}) \leq E(\mathbf{v}_{n+1}),\\
       \mathbf{v}_{n+1}, & \text{otherwise}.
    \end{cases} \label{eq:nonm_u}
\end{align}
Here $\mathbf{z}_{n+1}$ is computed by RAG-$\mathcal{N}$, i.e., $\mathbf{z}_{n+1}=\mathcal{R}_{\mathbf{w}_{n}}\left(-\alpha \nabla_{\mathcal{N}}E(\mathbf{w}_n)\right)$ with $\mathbf{w}_n=\rho(\hat{\mathbf{w}}_{n}) \hat{\mathbf{w}}_{n}$ defined in \eqref{eq:RNAG}, and $\mathbf{v}_{n+1} = \mathcal{R}_{\mathbf{u}_n}\left(-\alpha_0\beta^j \nabla_{\mathcal{N}}E(\mathbf{u}_n) \right)$ is obtained by selecting the smallest non-negative integer $j\geq 0$ satisfying the following nonmonotone Armijo condition \cite{CLXY2025SISC,Hongchao2004A} 
\begin{align}\label{eq:Armijo descent}
    E(\mathbf{v}_{n+1}) \leq C_n - \sigma \alpha_0\beta^j \|\nabla_{\mathcal{N}} E(\mathbf{u}_n)\|_H^2,
\end{align}
where $\alpha_0$ is a given initial trial step-size, $\beta \in (0,1)$ is a backtracking factor, and $C_n$ ($n\geq 1$) is computed iteratively as 
\begin{align}\label{eq:C_k}
   C_{n} = \frac{\varrho Q_{n-1} C_{n-1} + E(\mathbf{u}_{n})}{Q_{n}},\quad
    Q_{n} = \varrho Q_{n-1} +1,
\end{align} 
with given $\sigma, \varrho \in (0,1)$, $C_0 = E(\mathbf{u}_0)$ and $Q_0 = 1$. Finally, the implementation of the nmRAG-$\mathcal{N}$ algorithm is presented in Algorithm~\ref{al:modified RNAG}.

\begin{algorithm}[!t]
	\caption{nmRAG-\!$\mathcal{N}$ algorithm}
	\label{al:modified RNAG}
	
	In addition to the initial setting in Algorithm~\ref{al:RNAG}, initialize  $ \varrho, \beta, \sigma  \in (0,1)$, $\alpha_0>0$, $Q_0 = 1$ and $C_0 = E(\mathbf{u}_{0})$. Set $n=1$.   \\
	\While {stop criterion is not met} 
	{ 
    Compute $\hat{\mathbf{w}}_n = \mathbf{u}_n +t_n(\mathbf{u}_n - \mathbf{u}_{n-1})$ and $\mathbf{w}_n = \rho(\hat{\mathbf{w}}_n) \hat{\mathbf{w}}_n$ by the same procedures as in Algorithm~\ref{al:RNAG}.
    
    Compute $\mathbf{z}_{n+1}=\mathcal{R}_{\mathbf{w}_{n}}\left(-\alpha \nabla_{\mathcal{N}}E(\mathbf{w}_n)\right)$.   
   
    Compute $Q_{n}$ and $C_{n}$ by~\eqref{eq:C_k}. Set $\alpha_n = \alpha_0$.
    
    \While{ $E\left(\mathcal{R}_{\mathbf{u}_n}(-\alpha_n\nabla_{\mathcal{N}}E(\mathbf{u}_n))\right) > C_n - \sigma \alpha_n \|\nabla_{\mathcal{N}} E(\mathbf{u}_n) \|_H^2$ }
    {
    \vspace{2pt}
    
    $\alpha_n = \alpha_n \beta$
    }
    
    $\mathbf{v}_{n+1} = \mathcal{R}_{\mathbf{u}_n} (-\alpha_n \nabla_{\mathcal{N}}E(\mathbf{u}_n))$. 

    \vspace{2pt}
    
    \eIf{ $E(\mathbf{z}_{n+1}) \leq E(\mathbf{v}_{n+1})$ } 
    { 
    \vspace{2pt}
    
    $\mathbf{u}_{n+1} = \mathbf{z}_{n+1}$.
    }
    {
    \vspace{2pt} 
    
    $\mathbf{u}_{n+1} = \mathbf{v}_{n+1}$.
    }
    $n = n+1$.
	}
\end{algorithm}

\begin{remark}
Compared to Algorithm~\ref{al:RNAG}, Algorithm~\ref{al:modified RNAG} provides a better control over the functional value. In fact, the combination of \eqref{eq:nonm_u} and \eqref{eq:Armijo descent} states that  
\begin{align}\label{eq:nonmo}
    E(\mathbf{u}_{n+1}) \leq C_n - \sigma \alpha_n \|\nabla_{\mathcal{N}}E(\mathbf{u}_n)\|_H^2.
\end{align}
With this fact, one can verify that $E(\mathbf{u}_n) \leq C_n \leq C_{n-1}\leq C_0 =  E(\mathbf{u}_0)$, $n=1,\ldots$, and sequences $\{C_n\}$ and  $\{E(\mathbf{u}_n)\}$ converge to the same limit; see Lemmas 3.7 and 4.4 in \cite{CLXY2025SISC}. Furthermore, following the proof of \cite[Theorem~4.3]{CLXY2025SISC}, where~\eqref{eq:nonmo} plays a crucial role, it can be shown that every accumulation point of the sequence $\{\mathbf{u}_n\}$ generated by Algorithm~\ref{al:modified RNAG} is exactly a nontrivial solution of the coupled system~\eqref{eq:model}.
\end{remark}

\begin{remark}
We remark here that Algorithms~\ref{al:RNAG}-\ref{al:modified RNAG} are about the iteration of sequences on the infinite-dimensional Nehari manifold $\mathcal{N}$, leaving some subproblems including the computation of the RSD \eqref{eq:eta-SD}-\eqref{eq: R_grad_N} and the retraction~\eqref{eq:Nehari retraction}. There are various options for the numerical discretization of these subproblems, such as the finite element methods, finite difference methods or spectral methods. In our numerical test, we adopt the sine pseudospectral method (see Appendix~\ref{sec:Discre}), which has high efficiency and spectral accuracy and is applicable to the box domain with Dirichlet boundary conditions. 
\end{remark}

\begin{remark}
In our numerical test, the stopping criterion in Algorithms~\ref{al:RNAG}-\ref{al:modified RNAG} is taken as the following condition on the maximum residual of the coupled system \eqref{eq:model} (in the discrete level):
\begin{align}\label{eq:residual}
r_n:=\max_{1\leq i\leq m}\bigg\|- \varepsilon_i\Delta u_{ni}+a_i u_{ni}- \sum_{j=1}^{m}g_{ij}u_{nj}^2u_{ni}\bigg\|_{\infty} \leq \varepsilon_{\mathrm{tol}},
\end{align}
with a given tolerance $\varepsilon_{\mathrm{tol}}>0$. It can also be chosen as the examination on the Riemannian gradient or the $H$-gradient. 
\end{remark}

\section{Numerical results}\label{Sec:numerical result}

In this section, we firstly test the performance of the proposed RAG-$\mathcal{N}$ algorithm in Algorithm~\ref{al:RNAG} and  nmRAG-$\mathcal{N}$ algorithm in Algorithm~\ref{al:modified RNAG} with the RSD-$\mathcal{N}$ algorithm in~\eqref{eq:RSD}. Then we apply them to compute the ground state solutions of the coupled system~\eqref{eq:model} under different parameters and potentials. For $m=2$, the symmetry-breaking phenomenon of the ground state solutions of the coupled system~\eqref{eq:model} under symmetric/asymmetric Gaussian potentials, and spike-type ground state solutions of the coupled system~\eqref{eq:model} with small diffusion coefficients $\varepsilon_i$ are investigated. Moreover, semi-trivial ground state solutions of the coupled system~\eqref{eq:model} with $m=2$, $m=3$ and $m=4$ under constant potentials are also studied. In particular, for $m=2$, the parameter region corresponding to fully nontrivial and semi-trivial ground state solutions is further explored numerically in the coupling-coefficient space.

The initial guess is computed by $\mathbf{u}_0 = \rho(\hat{\mathbf{u}}_0)\hat{\mathbf{u}}_0$. For  $\hat{\mathbf{u}}_0$, we consider the following three types of initial data: 
\begin{itemize}
    \item $\hat{{u}}_{0i} ^{\mathrm{exp}}(\mathbf{x}) \;=  e^{-16(x^2+y^2)}, \quad \mathbf{x} =(x,y) \in \Omega, \quad i=1,2,\cdots,m, $ 
    \item $\hat{{u}}_{0i}^{\mathrm{sin}}(\mathbf{x}) \,\,\, = \sin (\pi(1-x^2))\sin (\pi(1-y^2)),\quad \mathbf{x} =(x,y) \in \Omega, \quad i=1,2,\cdots,m,$ 
    \item $\hat{{u}}_{0i}^{\mathrm{poly}}(\mathbf{x})  \,= (1-x^2)(0.5-x^2)(1-y^2)(0.5-y^2),\quad \mathbf{x} =(x,y) \in \Omega, \quad i=1,2,\cdots,m.$
\end{itemize} 
In the numerical test, unless specified, the computational domain is fixed as $\Omega = (-1,1)^2$, the initial data $\hat{\mathbf{u}}_0 = \hat{\mathbf{u}}_0^{\mathrm{exp}}$ is utilized (with zero boundary values enforced by sine interpolation), the mesh size $h=1/32$ in the sine pseudospectral discretization, and the stopping criterion \eqref{eq:residual} with the tolerance $\varepsilon_{\mathrm{tol}} = 10^{-6}$ is employed. For the nmRAG-$\mathcal{N}$ algorithm in Algorithm~\ref{al:modified RNAG}, the initial trial step-size $\alpha_0=0.1$, the parameters for the nonmonotone condition are selected as $\sigma = 10^{-3}$, $\varrho = 0.85$ and the backtracking coefficient $\beta = 0.25$.

\subsection{Performance comparisons  of the (nm)RAG-\texorpdfstring{$\mathcal{N}$}{N} and RSD-\texorpdfstring{$\mathcal{N}$}{N} algorithms} \label{subsec:ptest}

In this subsection, we compare the performance  of the RAG-$\mathcal{N}$ algorithm in Algorithm~\ref{al:RNAG} and the nmRAG-$\mathcal{N}$ algorithm in Algorithm~\ref{al:modified RNAG} with the RSD-$\mathcal{N}$ algorithm in~\eqref{eq:RSD} under a fixed step-size $\alpha_n=\alpha>0$ for computing the ground state solutions of the coupled system~\eqref{eq:model} in the following examples: 

\begin{description}
\item[{\bf Example 1:}] Fix $m=3$,  $\varepsilon_1 = \varepsilon_2 = \varepsilon_3=1$, $a_1 = a_2 = a_3 = 2(x^2 + y^2 + 1)$, $g_{11} = 2$, $g_{22} = 4$, $g_{33} = 6$ and $g_{12} = g_{13} = 4$ in the coupled system~\eqref{eq:model} 
\item[{\bf Example 2:}] Fix $m=4$, $\varepsilon_1 = \varepsilon_2 = \varepsilon_3 = \varepsilon_4=1$, $a_1 = a_2 = a_3 = a_4 =  2(x^2+y^2 +1)$, $g_{11} = 2$, $g_{22} = 4$, $g_{33} = 6$, $g_{44} = 8$ and  $g_{12} = g_{13} = g_{14} = g_{23} = g_{24} =  4$ in the coupled system~\eqref{eq:model}.
\item[{\bf Example 3:}] Fix $m=4$, $\varepsilon_1 = \varepsilon_2 = \varepsilon_3 = \varepsilon_4=1$, $a_1 = a_2 = a_3 = a_4 =  2(x^2+y^2 +1)$, $g_{11} = 2$, $g_{22} = 4$, $g_{33} = 6$, $g_{44} = 8$,  $g_{12} =  g_{23} = g_{24} =  4$,  $g_{13} = g_{14} = 2$, and $g_{34} = 8$ in the coupled system~\eqref{eq:model}.
\end{description}

Firstly, we numerically test the efficiency of the RAG-$\mathcal{N}$ algorithm in Algorithm~\ref{al:RNAG} compared with the RSD-$\mathcal{N}$ algorithm for computing the ground state solutions of the coupled system~\eqref{eq:model} in {\bf Examples 1-2}. Comparisons of the number of iterations and CPU time for reaching the stopping criterion $r_n \leq \varepsilon_{\mathrm{tol}}$ \eqref{eq:residual} are shown in Table~\ref{tab:RAG-RSD}. We can observe in Table~\ref{tab:RAG-RSD} that the RAG-$\mathcal{N}$ algorithm usually spends far fewer iterations than the RSD-$\mathcal{N}$ algorithm to reach the stopping criterion $r_n \leq \varepsilon_{\mathrm{tol}}$ \eqref{eq:residual}. Especially for {\bf Example 1} with $g_{23} = 6$ and {\bf Example 2} with $ g_{34} = 8$, when $\alpha = 0.1$, the RSD-$\mathcal{N}$ algorithm cannot reach the stopping condition within 100000 iterations, but the RAG-$\mathcal{N}$ algorithm only takes no more than 1000 iterations.  Regarding CPU time, although each iteration of RAG-$\mathcal{N}$ involves an additional nonlinear extrapolation compared with RSD-$\mathcal{N}$ and thus is slightly more expensive per iteration, the substantial reduction in iteration count leads to a significant decrease in overall CPU time. For instance, in {\bf Example 1} with $g_{23} = 6$ and $\alpha = 0.1$, RAG-$\mathcal{N}$ completes in 6.95 seconds, whereas RSD-$\mathcal{N}$ takes more than 1756.36 seconds. In addition, we can also observe that when the step-size $\alpha$ reduces from $\alpha = 0.1$ to $\alpha = 0.01$, the number of iterations required by the RSD-$\mathcal{N}$ algorithm is approximately inversely proportional to $\alpha$, while it  shows a slower growth  with the decrease of $\alpha$ for the RAG-$\mathcal{N}$ algorithm.

\begin{table}[!t]
	\renewcommand{\arraystretch}{1.5}
	\centering
    \footnotesize 
	\caption{Comparison of the RSD-$\mathcal{N}$ algorithm and the RAG-$\mathcal{N}$ algorithm in Algorithm~\ref{al:RNAG} for computing the ground states of the coupled system~\eqref{eq:model} in {\bf Examples 1-2} with the stopping criterion $r_n \leq \varepsilon_{\mathrm{tol}}$ \eqref{eq:residual}. 
    "iter" and "time" denote the number of the iterations and  the total CPU time in seconds (s), respectively.   }
	\vspace{2pt}
	\label{tab:RAG-RSD}
\begin{tabular}{clrrrrrr}
\hline
 & &\multicolumn{3}{c}{$\alpha = 0.1$} &\multicolumn{3}{c}{$\alpha = 0.01$}  \\  \cline{3-8} 
\multirow{-2}{*}{} &\multirow{-2}{*}{Method}&\multicolumn{1}{c}{iter} &\multicolumn{1}{c}{time (s)}  &\multicolumn{1}{c}{$r_n$ }&\multicolumn{1}{c}{iter} &\multicolumn{1}{c}{time (s)} &\multicolumn{1}{c}{$r_n$} \\ \hline
&RSD-$\mathcal{N}$  & 100000 & 1756.36 & 1.54$\times10^{-5}$ & 100000 & 749.92 &4.91$\times10^{-4}$ 
\\
\multicolumn{1}{c}{\multirow{-2}{*}{\begin{tabular}[c]{@{}c@{}}
{\bf Example 1} \\ with  $ g_{23} = 6$\end{tabular}} } & RAG-$\mathcal{N}$ & 825 &6.95 & 9.89$\times10^{-7}$ & 2727 &23.41 & 9.24$\times10^{-7}$   \\ \hline
&RSD-$\mathcal{N}$ & 316 & 2.53& 9.79$\times10^{-7}$ & 3231 &23.77 & 9.95$\times10^{-7}$ \\
\multicolumn{1}{c}{\multirow{-2}{*}{\begin{tabular}[c]{@{}c@{}}
{\bf Example 1} \\ with  $ g_{23} = 8$\end{tabular}} } &RAG-$\mathcal{N}$ & 382 &3.54 & 7.96$\times10^{-7}$ & 3100 &26.09& 9.93$\times10^{-7}$  \\ \hline
& RSD-$\mathcal{N}$ & 100000 &4780.84 & 1.54$\times10^{-5}$ & 100000 & 986.97 & 4.92$\times10^{-4}$  \\
\multicolumn{1}{c}{\multirow{-2}{*}{\begin{tabular}[c]{@{}c@{}}
{\bf Example 2}\\ with  $ g_{34} = 8$\end{tabular}} } & RAG-$\mathcal{N}$  & 932 &10.30   & 9.97$\times10^{-7}$ & 2853 &31.37  & 9.83$\times10^{-7}$ \\ \hline
&RSD-$\mathcal{N}$ & 352 &3.39 & 9.79$\times10^{-7}$ & 3576 &34.08 & 9.99$\times10^{-7}$  \\
\multicolumn{1}{c}{\multirow{-2}{*}{\begin{tabular}[c]{@{}c@{}}
{\bf Example 2} \\ with  $ g_{34} = 10$\end{tabular}} } &RAG-$\mathcal{N}$ & 357 &3.91 & 8.05$\times10^{-7}$ & 2934 &32.25 & 8.49$\times10^{-7}$  \\ \hline
\end{tabular}
\end{table}

\begin{figure}[!t]
	\centering
	\subfloat[$\alpha = 0.1$]{
    \includegraphics[width=0.32\textwidth,height=0.28\textwidth]{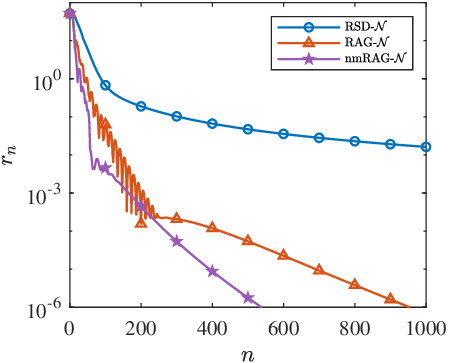}}
    \subfloat[$\alpha = 1$]{
    \includegraphics[width=0.32\textwidth,height=0.28\textwidth]{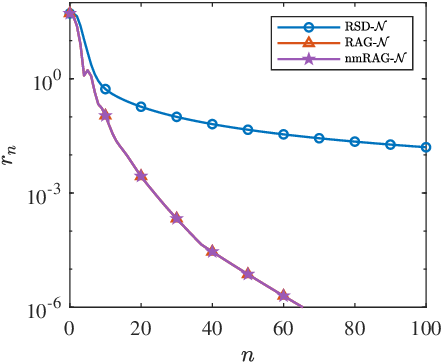}}
    \subfloat[$\alpha = 1.1$]{
    \includegraphics[width=0.32\textwidth,height=0.28\textwidth]{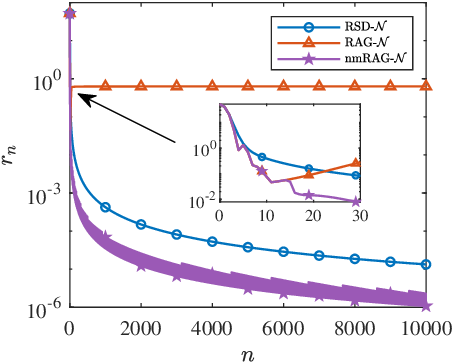}}
    \caption{The residuals $r_n$ in \eqref{eq:residual} w.r.t. $n$ of the RSD-$\mathcal{N}$,  RAG-$\mathcal{N}$ and nmRAG-$\mathcal{N}$ algorithms for computing the ground state solutions of the coupled system~\eqref{eq:model} in {\bf Example 3}  when $\alpha = 0.1$, $1$, $1.1$.}
	\label{fig:eff_comp1} 
\end{figure}

\begin{table}[!t]
\renewcommand{\arraystretch}{1.5}
\centering
\footnotesize
\caption{Influence of different initial data on the numerical results of the algorithms RSD-$\mathcal{N}$, RAG-$\mathcal{N}$, and nmRAG-$\mathcal{N}$.}
\label{tab:initial_sensitivity}
\begin{tabular}{cccccccc}
\hline
\multirow{2}{*}{Case} & \multirow{2}{*}{Method}
& \multicolumn{2}{c}{$\hat{\mathbf{u}}_0^{\mathrm{exp}}$}
& \multicolumn{2}{c}{$\hat{\mathbf{u}}_0^{\mathrm{sin}}$}
& \multicolumn{2}{c}{$\hat{\mathbf{u}}_0^{\mathrm{poly}}$} \\
\cline{3-8}
& & $E(\mathbf{u}_n)$ & $r_n$
  & $E(\mathbf{u}_n)$ & $r_n$
  & $E(\mathbf{u}_n)$ & $r_n$ \\
\hline
\multirow{3}{*}{\makecell{\textbf{Example 1} \\with $g_{23}=6$}}
& RSD-$\mathcal{N}$
& 3.157014 & $1.54\times10^{-5}$
& 3.157014 & $1.54\times10^{-5}$
& 3.157014 & $1.54\times10^{-5}$ \\

& RAG-$\mathcal{N}$
& 3.157014 & $9.89\times10^{-7}$
& 3.157014 & $9.96\times10^{-7}$
& 3.157014 & $9.94\times10^{-7}$ \\

& nmRAG-$\mathcal{N}$
& 3.157014 & $9.91\times10^{-7}$
& 3.157014 & $9.99\times10^{-7}$
& 3.157014 & $9.90\times10^{-7}$ \\
\hline

\multirow{3}{*}{\makecell{\textbf{Example 1} \\with $g_{23}=8$}}
& RSD-$\mathcal{N}$
& 2.841312 & $9.79\times10^{-7}$
& 2.841312 & $9.96\times10^{-7}$
& 2.841312 & $9.66\times10^{-7}$ \\

& RAG-$\mathcal{N}$
& 2.841312 & $7.96\times10^{-7}$
& 2.841312 & $9.17\times10^{-7}$
& 2.841312 & $4.13\times10^{-7}$ \\

& nmRAG-$\mathcal{N}$
& 2.841312 & $6.42\times10^{-7}$
& 2.841312 & $9.21\times10^{-7}$
& 2.841312 & $9.26\times10^{-7}$ \\
\hline

\multirow{3}{*}{\makecell{\textbf{Example 2} \\ with $g_{34}=8$}}
& RSD-$\mathcal{N}$
& 2.367760 & $1.54\times10^{-5}$
& 2.367760 & $1.54\times10^{-5}$
& 2.367760 & $1.54\times10^{-5}$ \\

& RAG-$\mathcal{N}$
& 2.367760 & $9.97\times10^{-7}$
& 2.367760 & $9.63\times10^{-7}$
& 2.367760 & $9.95\times10^{-7}$ \\

& nmRAG-$\mathcal{N}$
& 2.367760 & $9.99\times10^{-7}$
& 2.367760 & $7.77\times10^{-7}$
& 2.367760 & $1.00\times10^{-6}$ \\
\hline

\multirow{3}{*}{\makecell{\textbf{Example 2}\\ with $g_{34}=10$}}
& RSD-$\mathcal{N}$
& 2.185625 & $9.79\times10^{-7}$
& 2.185625 & $9.87\times10^{-7}$
& 2.185625 & $9.82\times10^{-7}$ \\

& RAG-$\mathcal{N}$
& 2.185625 & $8.05\times10^{-7}$
& 2.185625 & $2.51\times10^{-7}$
& 2.185625 & $6.29\times10^{-7}$ \\

& nmRAG-$\mathcal{N}$
& 2.185625 & $7.32\times10^{-7}$
& 2.185625 & $9.76\times10^{-7}$
& 2.185625 & $8.75\times10^{-7}$ \\
\hline
\end{tabular}
\end{table}

Moreover, to test the efficiency and the convergence of the nmRAG-$\mathcal{N}$ algorithm in Algorithm~\ref{al:modified RNAG}, Figure~\ref{fig:eff_comp1}  compares it with the RAG-$\mathcal{N}$  and RSD-$\mathcal{N}$ algorithms under different step sizes for computing the ground state solutions of the coupled system~\eqref{eq:model} in {\bf Example 3}. When $\alpha = 0.1$,  all three methods converge. As observed in the first plot of Figure~\ref{fig:eff_comp1}, nmRAG-N requires the fewest iterations to satisfy the stopping criterion $r_n \leq \varepsilon_{\mathrm{tol}}$ \eqref{eq:residual}. Compared to the RSD-$\mathcal{N}$ algorithm, the (nm)RAG-$\mathcal{N}$ algorithm exhibits a rapidly oscillating decline of the residual $r_n$ defined in \eqref{eq:residual} w.r.t. the number of the iteration $n$. Such oscillatory behavior is also observed in the classical Nesterov's accelerated gradient method for convex optimization problems in $\mathbb{R}^N$. When the step-size $\alpha = 1$, the plots of RAG-$\mathcal{N}$  and nmRAG-$\mathcal{N}$ algorithms  coincide. In nmRAG-$\mathcal{N}$ algorithm, the candidate iterates generated by the nonmonotone step-size search are not accepted, and the iteration reduces to the standard Nesterov-type update, i.e. $\mathbf{u}_{n+1} = \mathbf{z}_{n+1}$. When the step-size is increased to $\alpha = 1.1$, the RAG-$\mathcal{N}$ algorithm diverges, whereas nmRAG-$\mathcal{N}$ algorithm still converges and reaches the stopping criterion in fewer iterations than RSD-$\mathcal{N}$ algorithm. From the last plot in Figure~\ref{fig:eff_comp1}, we observe that the curves of RAG-$\mathcal{N}$ and nmRAG-$\mathcal{N}$ algorithm initially overlap. At $n=14$, a bifurcation occurs: the residual $r_n$ in RAG-$\mathcal{N}$ algorithm begins to increase, while in nmRAG-$\mathcal{N}$ algorithm it continues to exhibit an oscillatory decline. At this stage, the nonmonotone step-size search step~\eqref{eq:Armijo descent}  becomes active in nmRAG-$\mathcal{N}$ algorithm, and the new iterate is determined by $\mathbf{u}_{n+1} = \mathbf{v}_{n+1}$. This adaptive mechanism appears to be crucial in maintaining convergence beyond the stability threshold of the standard RAG-$\mathcal{N}$ algorithm. These numerical results indicate that nmRAG-$\mathcal{N}$ algorithm preserves the acceleration properties of RAG-$\mathcal{N}$ algorithm,  while exhibiting enhanced robustness with respect to the step-size selection, effectively enlarging the practical stability region of the accelerated iteration.

Since the proposed algorithms are essentially local minimization methods on the Nehari manifold, the computed solution may, in principle, depend on the choice of initial data, particularly in parameter regimes where multiple local minimizers may coexist. We therefore provide numerical tests to investigate the influence of initial data on  computed results. By using initial data $\hat{\mathbf{u}}_0^{\mathrm{exp}}$, $\hat{\mathbf{u}}_0^{\mathrm{sin}}$ and $\hat{\mathbf{u}}_0^{\mathrm{poly}}$, we compute the ground state solutions of \textbf{Examples 1} and \textbf{2} with all three proposed algorithms, namely RSD-$\mathcal{N}$, RAG-$\mathcal{N}$ and nmRAG-$\mathcal{N}$. Based on  the above numerical comparisons, we set $\alpha = 0.1$ for RAG-$\mathcal{N}$ and nmRAG-$\mathcal{N}$. The step-size in RSD-$\mathcal{N}$ is fixed as $\alpha_n  = 0.1$. The final functional values $E(\mathbf{u}_n)$ and residuals $r_n$ obtained with these initial data are reported in Table~\ref{tab:initial_sensitivity}. As shown in  Table~\ref{tab:initial_sensitivity}, despite the substantially different spatial profiles of the initial data, all three algorithms produce numerical states with essentially identical functional values for each of the tested cases. This consistency provides  numerical evidence for the reliability of the three algorithms in computing ground state solutions.

\subsection{The properties of the ground state solutions}

Based on the performance in Section~\ref{subsec:ptest}, we compute the ground state solutions of the coupled system~\eqref{eq:model} and explore their properties such as the symmetry-breaking phenomenon, semi-trivial ground state solutions and spike-type ground state solutions with small diffusion coefficients $\varepsilon_i$, by utilizing the nmRAG-$\mathcal{N}$ algorithm in Algorithm~\ref{al:modified RNAG} with $\alpha = 0.1 = \alpha_0$, $\sigma = 10^{-3}$, $\varrho = 0.85$ and $\beta = 0.25$.

\subsubsection{Symmetry-breaking ground states under symmetric/asymmetric Gaussian potentials}\label{subsec:Gaussian}

\begin{figure}[!t]
	\begin{minipage}{0.48\textwidth}
		\centering
		\subfloat[ $w_1 = w_2 = 1$]{\includegraphics[width=0.99\linewidth]{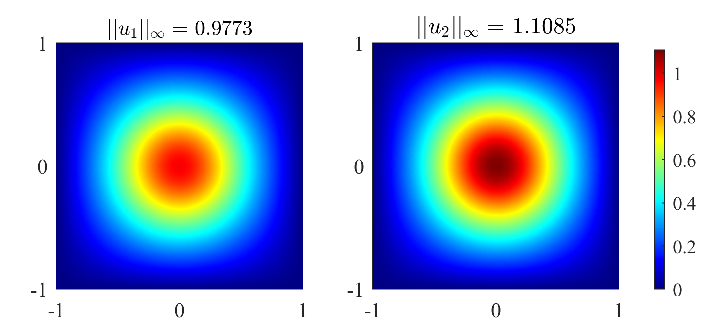}}\\
		\vspace{-1.5ex}
		\subfloat[$w_1 = w_2 = 8$]{\includegraphics[width=0.99\linewidth]{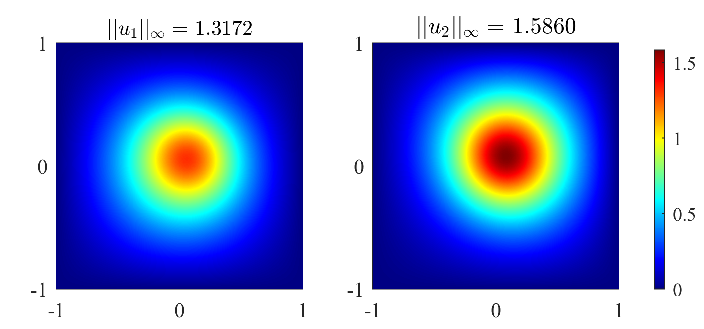}}\\
		\vspace{-1.5ex}
		\subfloat[ $w_1 = w_2 = 10$]{\includegraphics[width=0.99\linewidth]{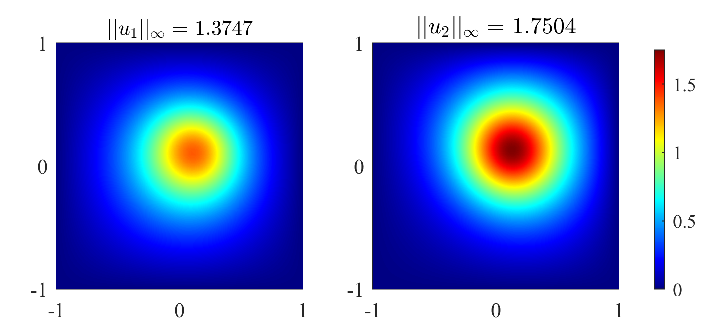}}
		\caption{Profiles of ground state solutions under the asymmetric Gaussian potential $\eqref{eq:Guass}$ for fixed $\varepsilon_1 = \varepsilon_2 = 1$, $g_{11} = 1$, $g_{12} = 10$, $ g_{22} = 3$, $(x_1^c,y_1^c) = (0.5,0.5)$, $(x_2^c,y_2^c) = (-0.5,-0.5)$ and $\delta_1 = \delta_2 = 1$, with different $w_i$.}
		\label{fig:Guass1}
	\end{minipage}
	\quad
   \begin{minipage}{0.48\textwidth}
   \centering
    \subfloat[$\delta_1 = \delta_2 = 1$]{\includegraphics[width=0.99\linewidth]{figure//prfi_ground//Guass//r0x1-0.5_r0x2-_0.5_r0y1-0.5_r0y2-_0.5//w0-10_delta-1.eps}}\\
    \vspace{-1.5ex}
    \subfloat[$\delta_1 = \delta_2 = 5$]{\includegraphics[width=0.99\linewidth]{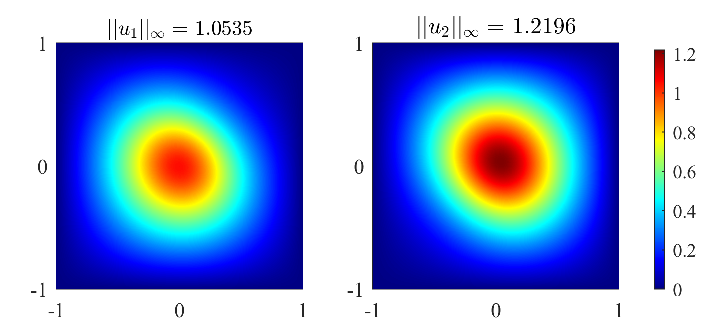}}\\
    \vspace{-1.5ex}
    \subfloat[$\delta_1 = \delta_2 = 100$.]{\includegraphics[width=0.99\linewidth]{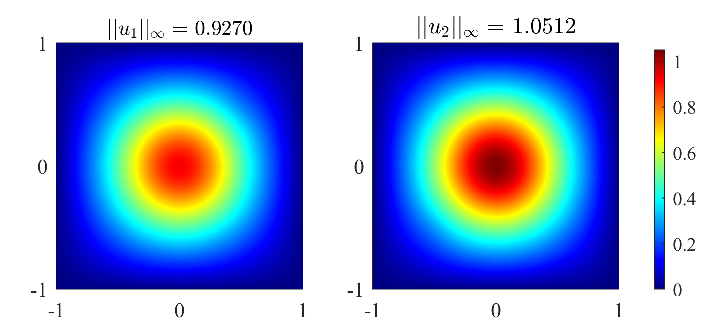}}
    \caption{Profiles of ground state solutions under the asymmetric Gaussian potential $\eqref{eq:Guass}$ for fixed $\varepsilon_1 = \varepsilon_2 = 1$, $g_{11} = 1$, $g_{12} = 10$, $g_{22} = 3$, $(x_1^c,y_1^c) = (0.5,0.5)$, $(x_2^c,y_2^c) = (-0.5,-0.5)$ and $w_1 = w_2 = 10$, with different $\delta_i$.}
    \label{fig:Guass2}
    \end{minipage}
\quad
\end{figure}

\begin{figure}[!t]
\begin{minipage}{0.48\textwidth}
\centering
    \subfloat[ $w_1 = w_2 = 5$]{\includegraphics[width=0.99\linewidth]{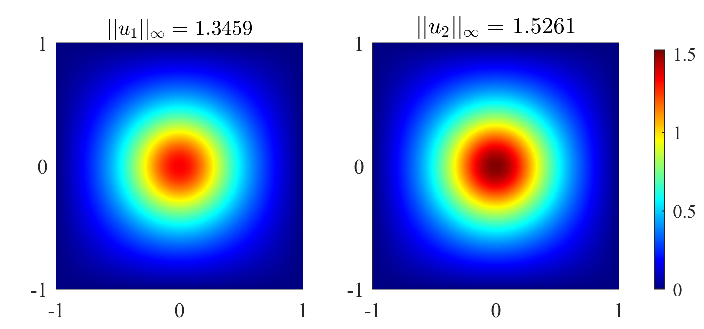}}\\
    \vspace{-1.5ex}
    \subfloat[$w_1 = w_2 = 6.5$]{\includegraphics[width=0.99\linewidth]{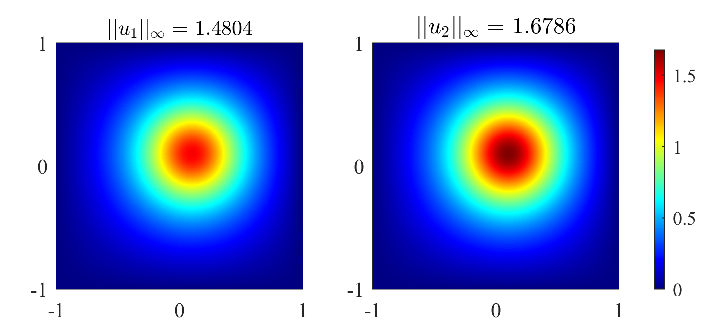}}\\
    \vspace{-1.5ex}
    \subfloat[ $w_1 = w_2 = 7$]{\includegraphics[width=0.99\linewidth]{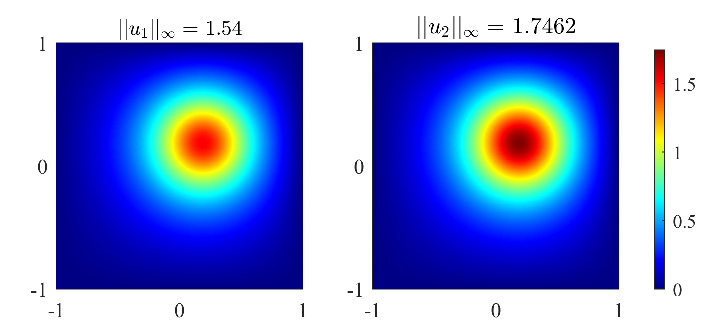}}
    \caption{Profiles of ground state solutions under the symmetric Gaussian potential $\eqref{eq:Guass}$ for fixed $\varepsilon_1 = \varepsilon_2 = 1$, $g_{11} = 1$, $g_{12} = 10$, $g_{22} = 3$, $(x_1^c,y_1^c) = (x_2^c,y_2^c) = (0,0)$,  and $\delta_1 = \delta_2 = 1$, with different $w_i$.}
    \label{fig:Guass4}
\end{minipage}
\quad
 \begin{minipage}{0.48\textwidth}
	\centering
	\subfloat[$\delta_1 = \delta_2 = 10$]{\includegraphics[width=0.99\linewidth]{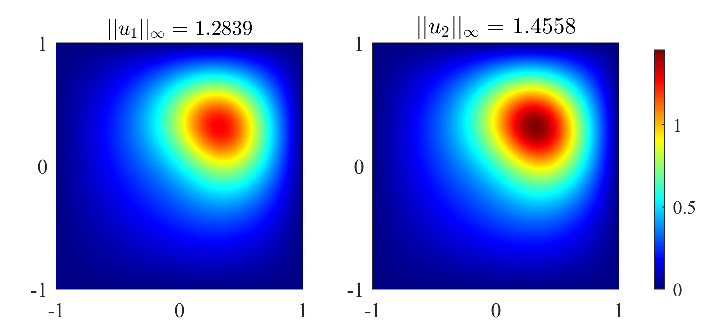}}\\
	\vspace{-1.5ex}
	\subfloat[$\delta_1 = \delta_2 = 30$]{\includegraphics[width=0.99\linewidth]{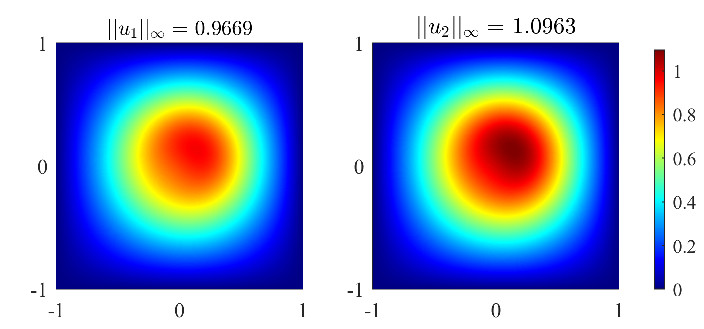}}\\
	\vspace{-1.5ex}
	\subfloat[$\delta_1 = \delta_2 = 35$.]{\includegraphics[width=0.99\linewidth]{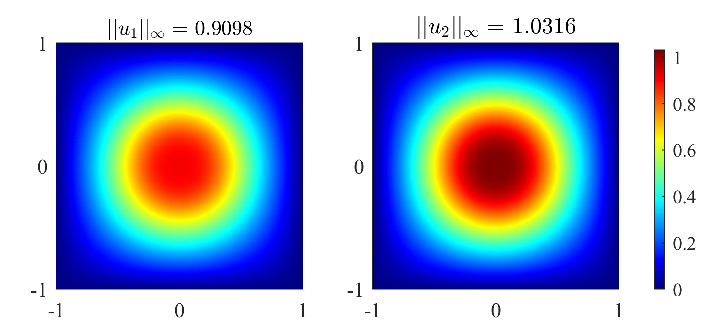}}
	\caption{Profiles of ground state solutions under the symmetric Gaussian potential $\eqref{eq:Guass}$ for fixed $\varepsilon_1 = \varepsilon_2 = 1$, $g_{11} = 1$, $g_{12} = 10$, $g_{22} = 3$, $(x_1^c,y_1^c) = (x_2^c,y_2^c) = (0,0)$ and $w_1 = w_2 = 10$, with different $\delta_i$.}
	\label{fig:Guass3}
\end{minipage}
\end{figure}

Note that when the potential is selected as the harmonic oscillator potential, the corresponding nonlinear Schr\"odinger system ~\eqref{eq:NLSE} can model the BEC trapped in the external potential. Here, we add the potential of a stirrer corresponding to a far-blue detuned Gaussian laser beam~\cite{Bao-Du2004,Jackson-PRL-1998}, 
 \begin{align}\label{eq:Guass}
     a_i = x^2+y^2 + w_i \exp({-\delta_i\left( (x - x_i^c)^2 + (y-y_i^c)^2 \right)}), \quad i=1,2,\ldots,m.
 \end{align}
For fixed $m=2$, $\varepsilon_1 = \varepsilon_2 = 1$, $g_{11} = 1$, $g_{12} = 10$ and $ g_{22} = 3$, we investigate the symmetry-breaking phenomenon of the ground state solutions of the coupled system~\eqref{eq:model} under symmetric/asymmetric Gaussian potentials with different coefficients ($x_i^c$, $y_i^c$, $w_i$ and $\delta_i$). Here, a function on the rectangular domain $\Omega = (-1,1)^2$ is said to be symmetric if it satisfies symmetry with respect to both coordinate axes and diagonals; otherwise, it is asymmetric. Since the coupled system~\eqref{eq:model} has asymmetric ground state solutions, the initial guess is 
generated by multiplying $\hat{\mathbf{u}}_0$ at each spatial grid point, by a random value uniformly sampled from $(0,1)$. Specifically, this is implemented using the MATLAB function ``rand''.

\begin{figure}[!t]
 \begin{minipage}{0.48\textwidth}
 	\centering
 	\subfloat[$ g_{12} = 1.8$]{\includegraphics[width=0.99\textwidth]{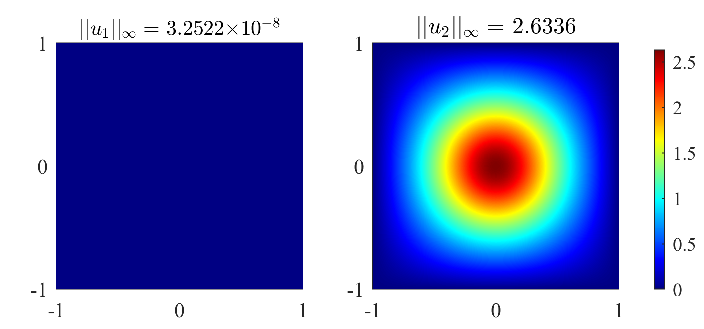}}
  \\
  \vspace{-1.8ex}
\subfloat[$ g_{12} = 2$]{\includegraphics[width=0.99\textwidth]{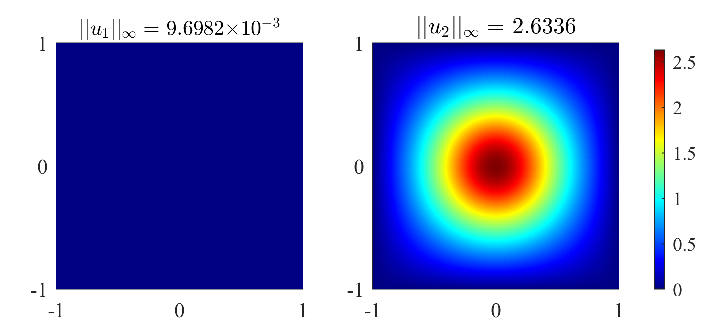}}
  \\
 \vspace{-1.8ex}
 	\subfloat[$ g_{12} = 2.2$]{\includegraphics[width=0.99\textwidth]{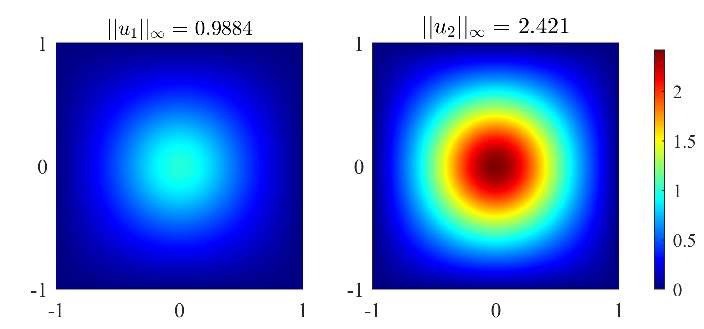}}
 \vspace{-1ex}
 	\caption{Profiles of the ground state solutions of the coupled system~\eqref{eq:model} when  $m=2$, $\varepsilon_1 \!=\! \varepsilon_2 \!=\! 1$, $\omega_1 \!=\!  \omega_2 \!=\!1$,  $ g_{11} \!=\! 1$ and $ g_{22} \!=\! 2$, with different $g_{12}$.}
 	\label{fig:m-2_eqv_V} 
    \end{minipage}
    \quad
\begin{minipage}{0.48\textwidth}
    \centering
 	\subfloat[$g_{12} = 1.6$]{\includegraphics[width=0.99\textwidth]{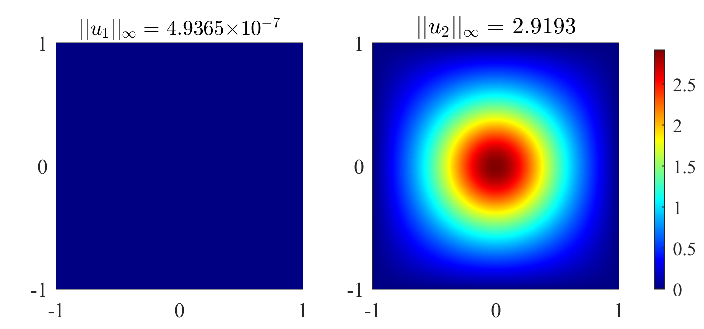}}
  \\
  \vspace{-1.8ex}
 	\subfloat[$g_{12} = 1.744$]{\includegraphics[width=0.99\textwidth]{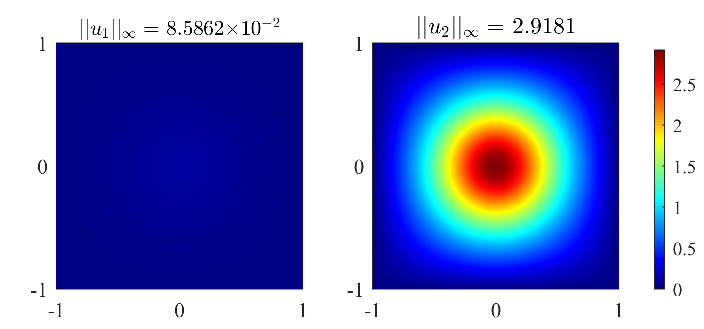}}
  \\
  \vspace{-1.8ex}
 	\subfloat[$g_{12} = 1.9$]{\includegraphics[width=0.99\textwidth]{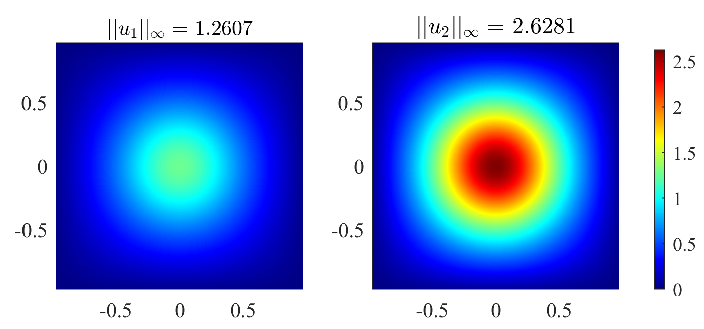}}
   \vspace{-1ex}
 	\caption{Profiles of the ground state solutions of the coupled system~\eqref{eq:model} when $m=2$, $\varepsilon_1 = \varepsilon_2 = 1$, $\omega_1 =1$, $ \omega_2 =2$,  $ g_{11} = 1$ and $ g_{22} = 2$, with different $g_{12}$. }
 	\label{fig:m-2_neqv_V1} 
 	\end{minipage}
\end{figure}

\begin{figure}[!t]
 \begin{minipage}{0.48\textwidth}
     \centering
     \subfloat{\includegraphics[width=1. \linewidth]{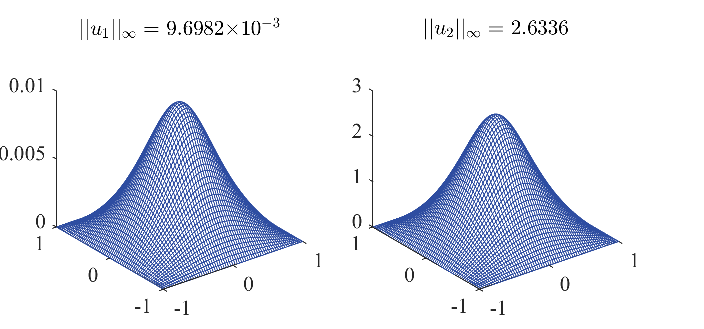}}
     \caption{ \hspace{-1em} Surface plot of the ground state solutions of the coupled system~\eqref{eq:model} when  $m=2$, $\varepsilon_1 \!=\! \varepsilon_2 \!=\! 1$, $\omega_1 \!=\! \omega_2 \!=\!1$, $ g_{11} \!=\! 1$, $ g_{22} \!=\! 2$ and $g_{12} \!=\! 2$. }
     \label{fig:another surf_plot1}
     \end{minipage}
     \quad
    \begin{minipage}{0.48\textwidth}
     \centering
     \subfloat{\includegraphics[width=1. \linewidth]{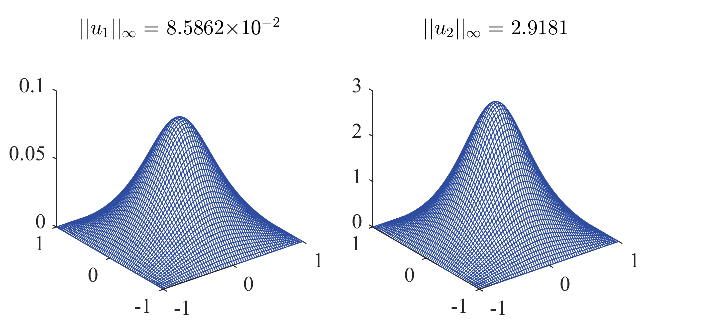}}
     \caption{ \hspace{-1em} Surface plot of the ground state solutions of the coupled system~\eqref{eq:model} when  $m=2$, $\varepsilon_1 \!=\! \varepsilon_2 \!=\! 1$, $\omega_1 =1, \omega_2 = 2$, $ g_{11} = 1$, $ g_{22} = 2$ and $g_{12}= 1.744$. }
     \label{fig:another surf_plot3}
     \end{minipage}
 \end{figure}

Firstly, choose $(x_1^c,y_1^c) = (0.5,0.5)$ and $(x_2^c,y_2^c) = (-0.5,-0.5)$, then the potentials $a_i$ are asymmetric. The profiles of the ground state solution of the coupled system~\eqref{eq:model} under asymmetric Gaussian potentials with different $w_i$ or $\delta_i$ are shown in Figures~\ref{fig:Guass1}-\ref{fig:Guass2}. For fixed $\delta_1 = \delta_2 = 1$, it can be observed from Figure~\ref{fig:Guass1} that the ground state solution of the coupled system~\eqref{eq:model} is symmetric when $w_1 = w_2 = 1$. As the $w_i$ increase to $10$, the peak of the ground state solution shifts from the center to the corner of the domain. The symmetry-breaking phenomenon happens and the critical value is around at $w_1=w_2 = 8$. Then, for fixed $w_1 = w_2=10$, Figure~\ref{fig:Guass2} shows that the ground state solution of the coupled system~\eqref{eq:model} is asymmetric when $\delta_1 =\delta_2  = 1$ and the peak shifts from the corner to the center of the domain as $\delta_1$ and $\delta_2$ increase to $100$. It indicates that the ground state solution changes from the asymmetric case to the symmetric case as the $\delta_i$ increase. Furthermore, choose {$(x_1^c,y_1^c) = (x_2^c,y_2^c) = (0,0)$} such that potentials $a_i$ are symmetric. The ground state solutions of the coupled system~\eqref{eq:model} under symmetric Gaussian potentials with different $w_i$ or $\delta_i$ are shown in Figures~\ref{fig:Guass4}-\ref{fig:Guass3}. It is observed that the symmetry-breaking phenomenon also happens even if the potentials are symmetric.

\subsubsection{Semi-trivial ground state solutions}

Let $\varepsilon_i  = 1$ and $a_i(\mathbf{x}) = \omega_i$ with the constants $\omega_i>0$ ($i=1,\ldots,m$). In this case, it is known that the coupled system~\eqref{eq:model} has semi-trivial solutions that have zero components.  It is widely concerned when the ground state solutions are semi-trivial in the literature \cite{ANTONIO_NLSE_jlms2007,Wang_Birfu2010,CORREIA2016,CORREIA2016JFA,Lin2005,MAIA2006JDE,SOAVE2016}.

Fix $m=2$. Denote $U_1$ and $U_2$ as the ground state solutions of the two single-component problems obtained by removing the coupling terms, namely,
\[
-\Delta U_1+\omega_1 U_1=g_{11}U_1^3,\qquad
-\Delta U_2+\omega_2 U_2=g_{22}U_2^3.
\]
They correspond to the two semi-trivial solutions $(U_1,0)$ and $(0,U_2)$ of the coupled system ~\eqref{eq:model}. Therefore, $\min\{E(U_1,0),E(0,U_2)\}$ serves as the least energy level among semi-trivial solutions. Figures~\ref{fig:m-2_eqv_V}-\ref{fig:another surf_plot3} show the profiles of the ground state solutions of the coupled system~\eqref{eq:model} with different $\omega_i$ and $g_{ij}$. On the one hand, when $\omega_1 = \omega_2 =1$, $g_{11}=1$ and $g_{22} = 2$, Figure~\ref{fig:m-2_eqv_V} shows that $\mathbf{u}_g$ is semi-trivial when $g_{12} = 1.8$, but fully nontrivial when $g_{12} = 2.2$. The middle plot in Figure~\ref{fig:m-2_eqv_V}  and  Figure~\ref{fig:another surf_plot1} indicate that $g_{12} =g_{22}= 2$ is close to a critical value at which the ground state solution changes from semi-trivial to fully nontrivial. On the other hand, when $1 = \omega_1 \neq \omega_2=2$, $g_{11} = 1$ and $g_{22} = 2$, Figures~\ref{fig:m-2_neqv_V1} and~\ref{fig:another surf_plot3} show that the ground state solution also transitions from semi-trivial to  fully nontrivial as $g_{12}$ increases, but the critical value of $g_{12}$ is approximately $1.744$ and lies strictly between $g_{11}$ and $g_{22}$. We further investigate the dependence of the numerical results on the initial data in parameter regimes close to the transition between semi-trivial and fully nontrivial states. For the two cases considered above, we select several values of $g_{12}$ near the observed transition behavior and compute the ground state solutions using the three different initial data $\hat{\mathbf{u}}_0^{\mathrm{exp}}$, $\hat{\mathbf{u}}_0^{\mathrm{sin}}$ and $\hat{\mathbf{u}}_0^{\mathrm{poly}}$. As shown in Table~\ref{tab:initial_near_transition}, even for these parameter values close to the numerically observed transition points, the energies obtained from the three substantially different initial data agree to six decimal places. We have also compared the corresponding numerical solution profiles and found no visible differences at the resolution of our computations. For conciseness, these profiles are not shown here. These results provide  numerical support that the computed solutions correspond to ground states, and further strengthen the evidence for the observed transition from semi-trivial to fully nontrivial ground state solutions.

\begin{table}[!t]
\renewcommand{\arraystretch}{1.5}
\centering
\footnotesize
\setlength{\tabcolsep}{4pt}
\caption{Numerical results obtained by nmRAG-$\mathcal{N}$ with different initial data for $g_{12}$ near the critical values, with $\omega_1=1$, $g_{11}=1$, and $g_{22}=2$.}
\label{tab:initial_near_transition}
\begin{tabular}{cc cc cc cc}
\hline
\multirow{2}{*}{$\omega_2$}
& \multirow{2}{*}{$g_{12}$}
& \multicolumn{2}{c}{$\hat{\mathbf{u}}_0^{\mathrm{exp}}$}
& \multicolumn{2}{c}{$\hat{\mathbf{u}}_0^{\mathrm{sin}}$}
& \multicolumn{2}{c}{$\hat{\mathbf{u}}_0^{\mathrm{poly}}$} \\
\cline{3-8}
&
& $E(\mathbf{u}_n)$ & $r_n$
& $E(\mathbf{u}_n)$ & $r_n$
& $E(\mathbf{u}_n)$ & $r_n$ \\
\hline

\multirow{3}{*}{$1$}
& 1.99
& 6.603390 & $8.52\times10^{-7}$
& 6.603390 & $9.93\times10^{-7}$
& 6.603390 & $9.64\times10^{-7}$ \\

& 2.00
& 6.603390 & $9.97\times10^{-7}$
& 6.603390 & $9.87\times10^{-7}$
& 6.603390 & $9.98\times10^{-7}$ \\

& 2.01
& 6.603066 & $9.17\times10^{-7}$
& 6.603066 & $9.79\times10^{-7}$
& 6.603066 & $8.88\times10^{-7}$ \\
\hline
\multirow{3}{*}{$2$}
& 1.740
& 8.705321 & $1.52\times10^{-7}$
& 8.705321 & $8.98\times10^{-7}$
& 8.705321 & $9.86\times10^{-7}$ \\

& 1.744
& 8.705323 & $9.47\times10^{-7}$
& 8.705323 & $4.49\times10^{-7}$
& 8.705323 & $8.66\times10^{-7}$ \\

& 1.750
& 8.704979 & $7.95\times10^{-7}$
& 8.704979 & $9.65\times10^{-7}$
& 8.704979 & $9.83\times10^{-7}$ \\
\hline
\end{tabular}
\vspace{-1ex}
\end{table}

\begin{figure}[!t]
    \centering
    \includegraphics[width=0.31\textwidth,height=0.28\textwidth]{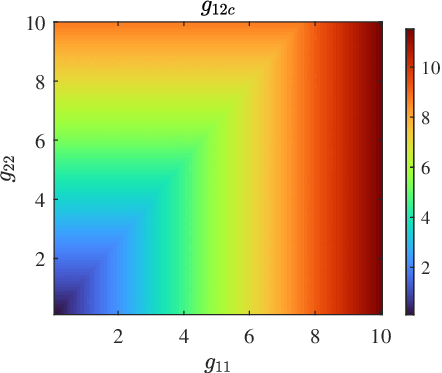}\hfill
   \includegraphics[width=0.31\textwidth,height=0.28\textwidth]{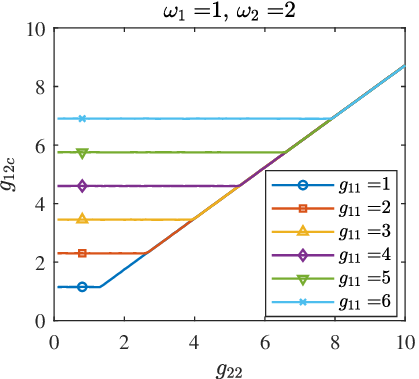}\hfill
    \includegraphics[width=0.31\textwidth,height=0.28\textwidth]{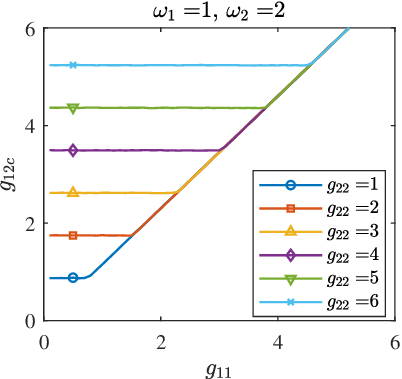}
    \caption{The heat map of the numerically computed $g^c_{12}$  w.r.t. $g_{11}$ and $g_{22}$ when $\omega_1=1$ and $\omega_2=2$ (left), together with the corresponding slices in the $g_{11}$ (middle) and $g_{22}$ (right) direction, respectively. }
    \label{fig:g_12c 1}
\end{figure}

\begin{figure}[!t]
    \centering
    \includegraphics[width=0.7\textwidth,height=0.56\textwidth]{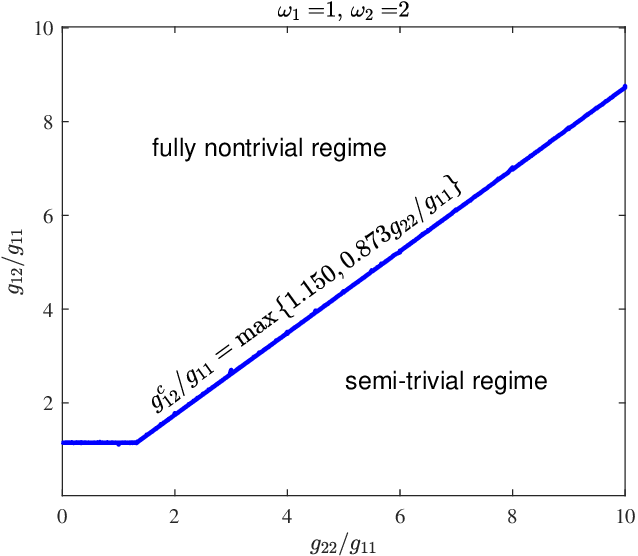}
    \caption{Phase diagram of the ground state solution w.r.t. the coupling coefficients $g_{11}, g_{12}$ and $g_{22}$ when $\omega_1=1$ and $\omega_2 = 2$. The solid curve shows the numerically computed critical value $g^c_{12}/g_{11}$ versus $g_{22}/g_{11}$, with the fitting result given by $g^c_{12}/g_{11} = \max\{1.150,0.873g_{22}/g_{11}\}$. Above the curve, the ground state solution is fully nontrivial; below the curve, it is  semi-trivial.} 
    \label{fig:g_12c 2}
\end{figure}

We remark here that for unbounded domain $\Omega = \mathbb{R}^N$ ($N=2$, $3$) and $m=2$, Ambrosetti and Colorado have proved in \cite{ANTONIO_NLSE_jlms2007} that there exist $\Gamma$ and $\Gamma'$ with $\Gamma'>\Gamma >0$ s.t. the ground state solutions are fully nontrivial  when $ g_{12}>\Gamma'$ and semi-trivial when $g_{12} < \Gamma$. The above numerical observations  suggest that a similar transition phenomenon for unbounded domain  in \cite{ANTONIO_NLSE_jlms2007} also holds in the bounded-domain case, of which a theoretical verification remains open to the best of our knowledge. 
Motivated by the above numerical observations in Figures~\ref{fig:m-2_eqv_V}-\ref{fig:another surf_plot3}, we further carefully study the critical value of $g_{12}$, denoted as $g^c_{12}$, w.r.t. $g_{11}$ and $g_{22}$,  for the transition phenomenon between semi-trivial and fully nontrivial ground state solutions under fixed $\omega_1$ and $\omega_2$ in the bounded-domain case. To distinguish the semi-trivial and fully nontrivial regimes, we introduce an indicator for the numerically computed ground state solution $\mathbf{u}_g$ as
\[
e_{\mathrm{rel}} = \frac{\bigl|E(\mathbf{u}_g)-\min\{E(U_1,0),E(0,U_2)\}\bigr|}{\bigl|\min\{E(U_1,0),E(0,U_2)\}\bigr|}.
\]
Then $\mathbf{u}_g$ is regarded as fully nontrivial whenever $e_{\mathrm{rel}}>10^{-6}$ and semi-trivial otherwise, where the threshold $10^{-6}$ is chosen in accordance with the stopping tolerance of the iterative algorithm to avoid numerical misclassification. We approximate $g^c_{12}$ by the bisection method, terminating when the interval length falls below the considered tolerance $\varepsilon_{\mathrm{bis}} = 10^{-2}$.

\begin{table}[!t]
\renewcommand{\arraystretch}{1.5}
\centering
\footnotesize
\caption{Sensitivity of the numerically determined critical coupling coefficient
$g_{12}^{c}$ to the initial data $\hat{\mathbf{u}}_0$, spatial mesh size $h$, and bisection tolerance $\varepsilon_{\mathrm{bis}}$
for $\omega_1=1$, $\omega_2=2$, and $g_{11}=1$.}
\label{tab:sensitivity}
\begin{tabular}{lccc}
\hline
 
& $g_{12}^c (g_{22}=1)$
& $g_{12}^c (g_{22}=2)$
& $g_{12}^c (g_{22}=5)$ \\
\hline
$\hat{\mathbf{u}}_0=\hat{\mathbf{u}}_0^{\mathrm{exp}},\;\;\;
 h=1/32,\;
 \varepsilon_{\mathrm{bis}}=10^{-2}$
& 1.151539 & 1.752871 & 4.362682 \\

$\hat{\mathbf{u}}_0=\hat{\mathbf{u}}_0^{\mathrm{sin}},\;\;\;\;\;
 h=1/32,\;
 \varepsilon_{\mathrm{bis}}=10^{-2}$
& 1.151539 & 1.752871 & 4.362682 \\

$\hat{\mathbf{u}}_0=\hat{\mathbf{u}}_0^{\mathrm{poly}}, \,
 h=1/32,\;
 \varepsilon_{\mathrm{bis}}=10^{-2}$
& 1.151539 & 1.752871 & 4.362682 \\
\hline

$\hat{\mathbf{u}}_0=\hat{\mathbf{u}}_0^{\mathrm{exp}},\;\;\;
 h=1/16,\;
 \varepsilon_{\mathrm{bis}}=10^{-2}$
& 1.151539 & 1.752871 & 4.362682 \\

$\hat{\mathbf{u}}_0=\hat{\mathbf{u}}_0^{\mathrm{exp}},\;\;\;
 h=1/32,\;
 \varepsilon_{\mathrm{bis}}=10^{-2}$
& 1.151539 & 1.752871 & 4.362682 \\

$\hat{\mathbf{u}}_0=\hat{\mathbf{u}}_0^{\mathrm{exp}},\;\;\;
 h=1/64,\;
 \varepsilon_{\mathrm{bis}}=10^{-2}$
& 1.151539 & 1.752871 & 4.362682 \\
\hline

$\hat{\mathbf{u}}_0=\hat{\mathbf{u}}_0^{\mathrm{exp}},\;\;\;
 h=1/32,\;
 \varepsilon_{\mathrm{bis}}=10^{-2}$
& 1.151539 & 1.752871 & 4.362682 \\

$\hat{\mathbf{u}}_0=\hat{\mathbf{u}}_0^{\mathrm{exp}},\;\;\;
 h=1/32,\;
 \varepsilon_{\mathrm{bis}}=10^{-3}$
& 1.149645 & 1.744290 & 4.362682 \\
$\hat{\mathbf{u}}_0=\hat{\mathbf{u}}_0^{\mathrm{exp}},\;\;\;
 h=1/32,\;
 \varepsilon_{\mathrm{bis}}=10^{-4}$
& 1.149803 & 1.744577 & 4.362628 \\
\hline
\end{tabular}
\end{table}

The numerical results of $g^c_{12}$ w.r.t. $g_{11}$ and $g_{22}$ for fixed $\omega_1 = 1$ and $\omega_2 = 2$ are shown in Figure~\ref{fig:g_12c 1}. 
The left plot in Figure~\ref{fig:g_12c 1} gives a heat map of the critical threshold $g^c_{12}$ in the $(g_{11},g_{22})$-plane. It suggests a rough division into two regions. In the upper-left region, $g^c_{12}$ varies mainly with $g_{22}$, whereas in the lower-right region, it varies mainly with $g_{11}$.  To better understand how $g^c_{12}$ varies with $g_{11}$ and $g_{22}$,  we further plot the slices of the heat map along these two directions, as shown in the middle and right plots in Figure~\ref{fig:g_12c 1}. It is observed that, for each fixed $g_{11}$, $g^c_{12}$ is nearly constant for relatively small $g_{22}$, while it increases approximately linearly with respect to relatively large $g_{22}$, and the curves for different $g_{11}$ are nearly overlapping. Similar behaviors of $g^c_{12}$ with respect to $g_{11}$ can be observed for different fixed values of $g_{22}$.

To further investigate the quantitative dependence of $g^c_{12}$ on $g_{11}$ and $g_{22}$, we plot the phase diagram of the ground state solution with boundary phase curve ${g^c_{12}}/{g_{11}}$ versus ${g_{22}}/{g_{11}}$ in  Figure~\ref{fig:g_12c 2}. It shows that the observed two-region pattern in Figure~\ref{fig:g_12c 1} can be compressed into a single and clear phase boundary, which  is  well described by a  piecewise curve. Above this curve, the system prefers the coexistence of the two components in ground state solutions. Below the curve, by contrast, the system prefers a semi-trivial ground state solution. The shape of boundary phase curve  motivates the three-parameter fitting model $g^c_{12}/g_{11} =\max\{a g_{22}/g_{11}+b,c\}$. The fitting result gives $a=0.873$, $b=0$ and $c=1.150$ with the mean squared error $\mathrm{MSE}=1.408\times10^{-5}$, and the fitted critical coupling coefficient reads
\[ g^c_{12}\approx \max\{1.150\,g_{11},\;0.873\,g_{22}\}. \]

We then examine the sensitivity of the numerically determined critical coupling coefficient $g_{12}^c$ with respect to the choice of initial data, the spatial mesh size $h$, and the bisection tolerance $\varepsilon_{\mathrm{bis}}$ in Table~\ref{tab:sensitivity}. The results show that i) the numerically determined values of $g_{12}^c$ are essentially insensitive to the choice of initial data and the spatial mesh sizes considered here, with the corresponding values agreeing within the reported numerical precision; ii) reducing the bisection tolerance $\varepsilon_{\mathrm{bis}}$ gradually improves the accuracy of $g_{12}^c$ and makes the numerically computed critical values increasingly stable. However, although smaller bisection tolerances would yield more accurate values, they would also increase the computational cost. The choice $\varepsilon_{\mathrm{bis}}=10^{-2}$ provides a reasonable balance between numerical accuracy and computational efficiency for determining the critical coupling coefficients used in the phase diagram. Moreover, the tests with smaller tolerances further confirm that the qualitative location and trend of the critical transition are stable with respect to the bisection tolerance. Overall, these sensitivity studies support the reliability of the phase diagram and the fitted critical curve reported above.

Furthermore, we also show the profiles of ground state solutions of the coupled system~\eqref{eq:model} for $m=3$ and $m=4$ in Figure~\ref{fig:m-3_sol} and Figure~\ref{fig:m-4_sol}, respectively. It can be observed that as $g_{ij}$ increase, the number of nontrivial components in ground state solution increases accordingly and  the ground state solution eventually  changes from semi-trivial to fully nontrivial.

The above numerical findings provide insights for further theoretical understanding of semi-trivial ground state solutions.

\begin{figure}[!t]
	\centering
	\subfloat[$ g_{12}  = g_{13} =  g_{23} = 2$]{\includegraphics[width=0.68\textwidth]{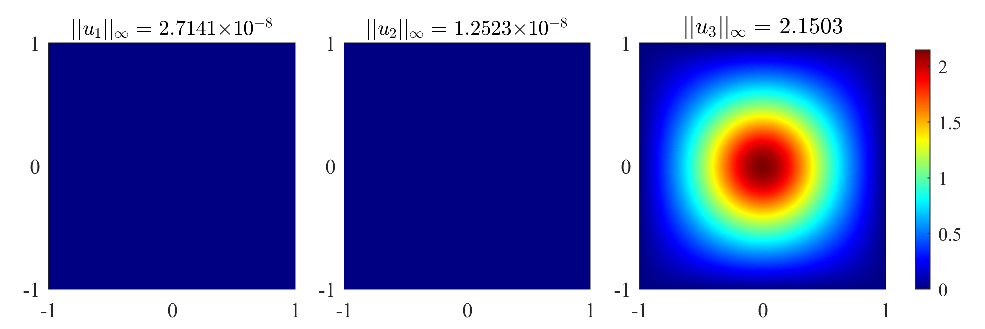}}\\
    \vspace{-1.5ex}
 	\subfloat[$ g_{12}  =  g_{13} = 2,  g_{23} = 10$]{\includegraphics[width=0.68\textwidth]{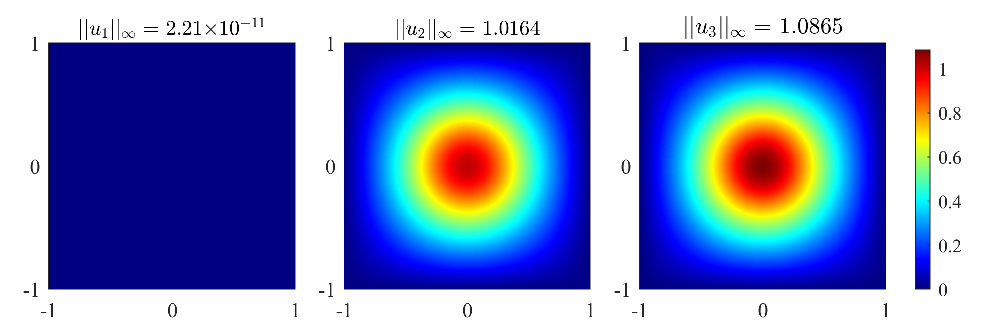}}\\
    \vspace{-1.5ex}
 	\subfloat[$ g_{12} =  g_{13} =  g_{23} = 10$]{\includegraphics[width=0.68\textwidth]{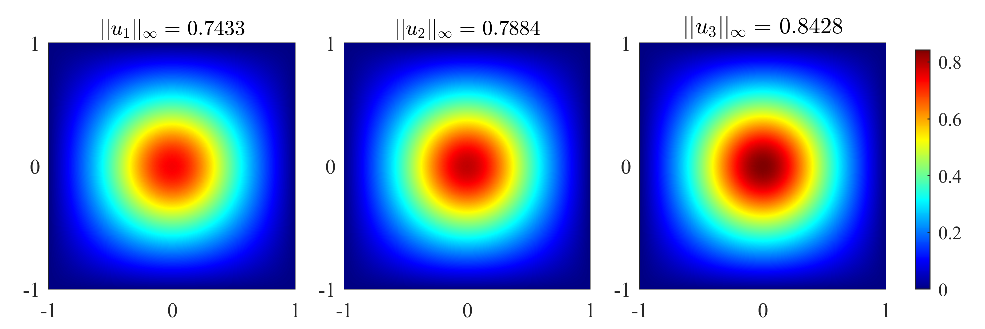}}
	\caption{Profiles of the ground state solutions of the coupled system~\eqref{eq:model} with $m=3$, when $\varepsilon_i  = \omega_i = 1$  $(i=1, 2, 3)$, $g_{11} = 1$, $g_{22} = 2$ and $ g_{33} = 3$.}
	\label{fig:m-3_sol} 
\end{figure}

\subsubsection{Spike-type ground state solutions with small diffusion coefficients \texorpdfstring{$\varepsilon_i$}{varepsiloni}}

Finally, we compute ground state solutions of the coupled system~\eqref{eq:model} with small diffusion coefficients $\varepsilon_i$ in this subsection, which leads to  a singularly perturbed problem. For $m=2$, choose the harmonic oscillator potential
\begin{align}\label{eq:poten_2}
    a_1(\mathbf{x}) = a_2(\mathbf{x}) = x^2 +y^2 +1. 
 \end{align}
Since the solutions are sharply peaked for small $\varepsilon_i$, we discretize the subproblems \eqref{eq: R_grad_N} by the sine pseudospectral method with a smaller mesh-size $1/64$. The results for $\varepsilon_i = 0.1$, $0.01$, $0.001$ with $g_{11} = 1$, $g_{12} = 10$ and  $g_{22} = 3$ are plotted in Figure~\ref{fig:perturbed1}. It can be observed that, as $\varepsilon_i$ decreases, the ground state becomes a spike-type solution, with its peak becoming sharper and its maximum value decreasing.

\section{Concluding remarks}\label{sec:concluding}

\begin{figure}[!t]
	\centering
	\vspace{-1ex}
	\subfloat[$ g_{12} =  g_{13} =  g_{14} = 2$, $ g_{23} =  g_{24} =  3$ and $ g_{34} = 2$]{\includegraphics[width=0.99\textwidth]{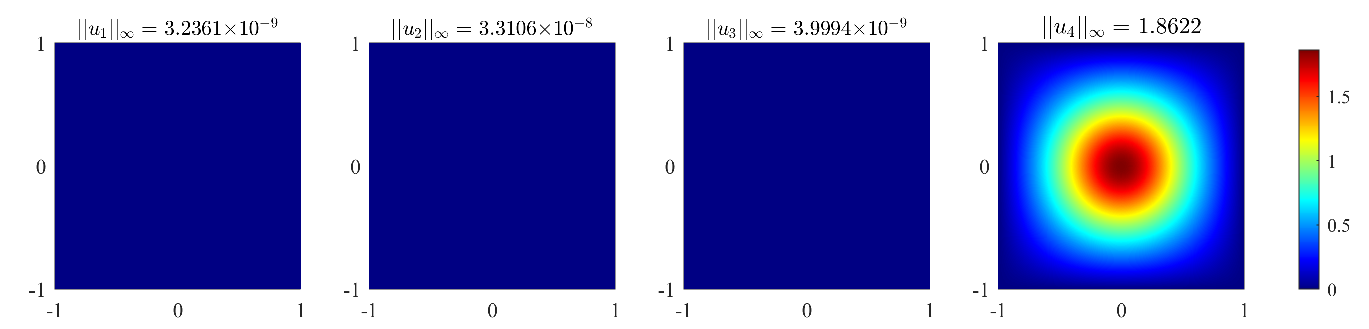}}\\
    \vspace{-1.8ex}
 	\subfloat[$ g_{12} =  g_{13} =  g_{14} = 10$, $ g_{23} =  g_{24} =3$ and $ g_{34} = 2$]{\includegraphics[width=0.99\textwidth]{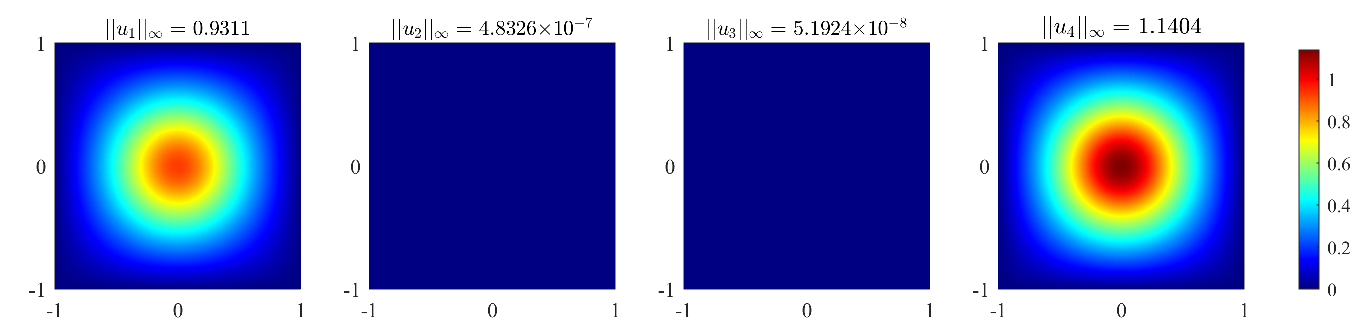}}\\
    \vspace{-1.8ex}
 	\subfloat[$ g_{12} =  g_{13} =  g_{14} = 10$, $ g_{23} =  g_{24} = 10$ and $ g_{34} = 2$]{\includegraphics[width=0.99\textwidth]{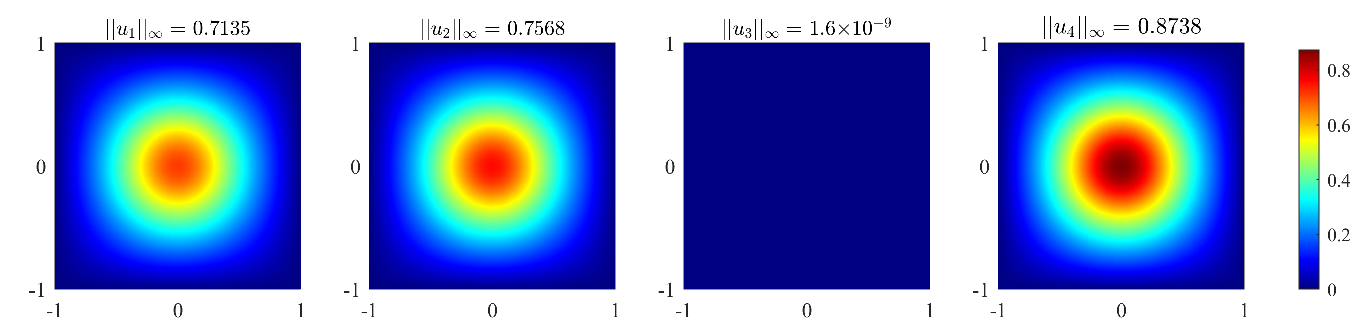}}\\
    \vspace{-1.8ex}
    \subfloat[ $ g_{12} =  g_{13} =  g_{14} = 10$, $ g_{23} =  g_{24} = 10$, $ g_{34} = 10$]{\includegraphics[width=0.99\textwidth]{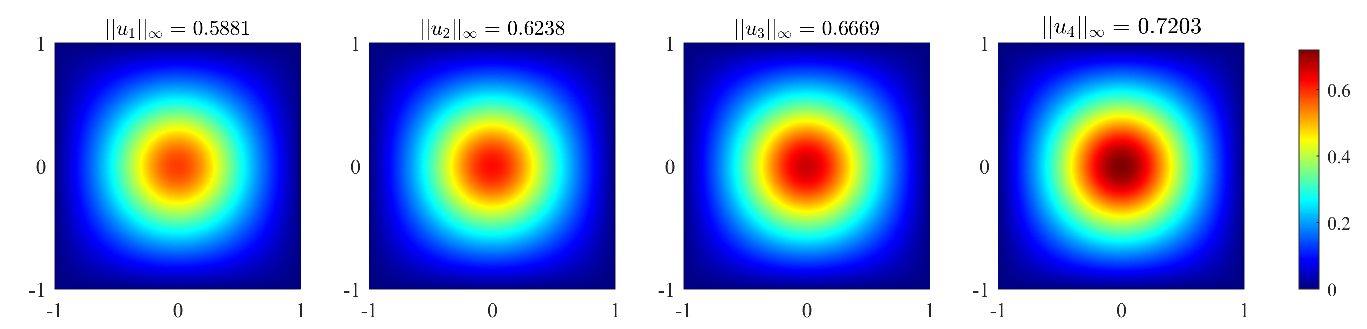}}
    \vspace{-1ex}
		\caption{Profiles of the ground state solutions of the coupled system~\eqref{eq:model} with $m=4$, when $\varepsilon_i = \omega_i =1 (i=1, 2, 3, 4)$, $ g_{11} = 1$, $ g_{22}  = 2$,  $g_{33} = 3$ and $ g_{44} = 4$.}
	\label{fig:m-4_sol} 
    \vspace{-.5ex}
\end{figure}

In this paper, we proposed two efficient Riemannian accelerated gradient algorithms for computing the ground state solutions of a class of coupled semilinear elliptic systems~\eqref{eq:model} based on the Nehari manifold optimization method. Inspired by the Nesterov accelerated gradient method, an easy-to-implement nonlinear extrapolation on $\mathcal{N}$ was incorporated into the Riemannian steepest-descent algorithm of NMOM, and the nonmonotone step-size search rule was further embedded. Numerical tests demonstrated that the proposed algorithms are significantly more efficient than the traditional Riemannian steepest-descent algorithm. In particular, we explored the properties of the ground state solutions under different parameter settings, revealing interesting phenomena that differ from those observed in single-equation cases. It is worth mentioning that the current work is applicable to compute the ground state solutions whose $\mathrm{MI}=1$. In the future, it would be desirable to extend these methods to compute the multiple solutions with higher Morse index.

\section*{Acknowledgment}
\vspace{-.5ex}
The work of Z.~Chen and Z.~Xie was supported by the NSFC grant 12171148 and the Major Program of Xiangjiang Laboratory, China (No.~22XJ01013). The work of W.~Liu was supported by NSFC grant 12571448 and the Innovation Research Foundation of NUDT (No.~202402-YJRC-XX-002). The work of W.~Yi was supported by the NSFC grant 12471374, Hunan Provincial Natural Science Foundation (No.~2025JJ40001), and the Key Scientific Research Project of the Education Department of Hunan Province (No.~23A0034). Part of this work was done when W.~Yi was visiting Department of Mathematics, National University of Singapore under the support of the China Scholarship Council.

  \begin{figure}[!t]
  \centering
 		\subfloat[$\varepsilon_1 = 
    \varepsilon_2 = 0.1$]{\includegraphics[width=0.48\linewidth]{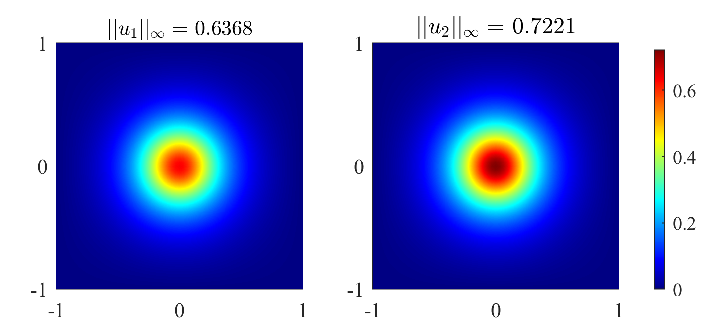}}
 		\quad\;
 		\subfloat[$\varepsilon_1 = 
    \varepsilon_2 = 0.1$]{\includegraphics[width=0.48\linewidth]{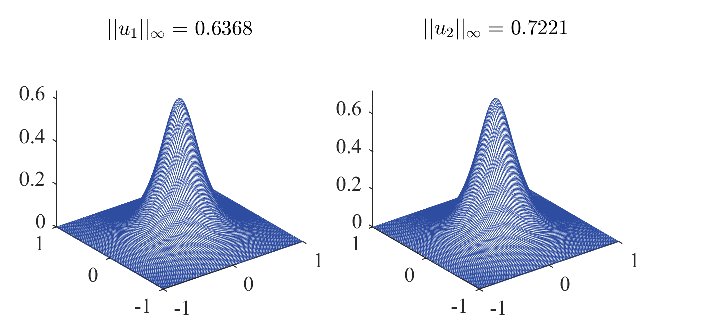}} \\
 \vspace{-1.8ex}
 		\subfloat[$\varepsilon_1 = 
    \varepsilon_2 = 0.01$]{\includegraphics[width=0.48\linewidth]{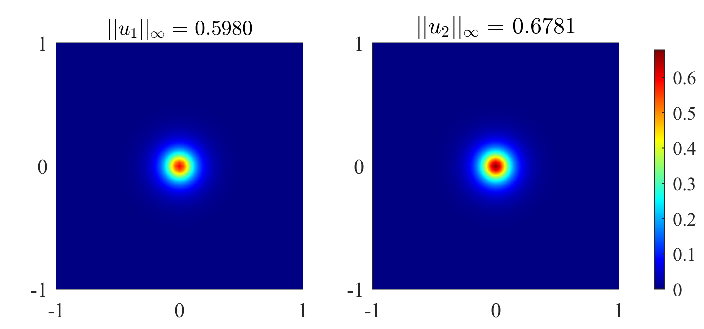}}
 		\quad \;
 		\subfloat[$\varepsilon_1 = 
    \varepsilon_2 = 0.01$]{\includegraphics[width=0.48\linewidth]{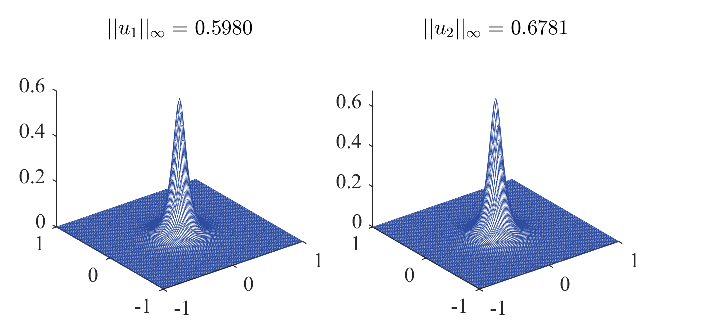}}\\
 	\vspace{-1.8ex}
 		\subfloat[$\varepsilon_1 = 
    \varepsilon_2 = 0.001$]{\includegraphics[width=0.48\linewidth]{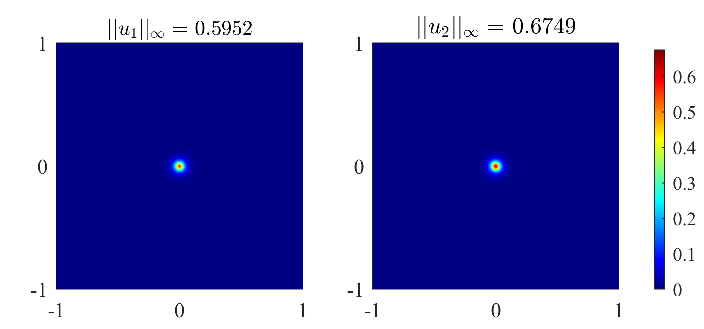}}
 		\quad \;
 		\subfloat[$\varepsilon_1 = 
    \varepsilon_2 = 0.001$]{\includegraphics[width=0.48\linewidth]{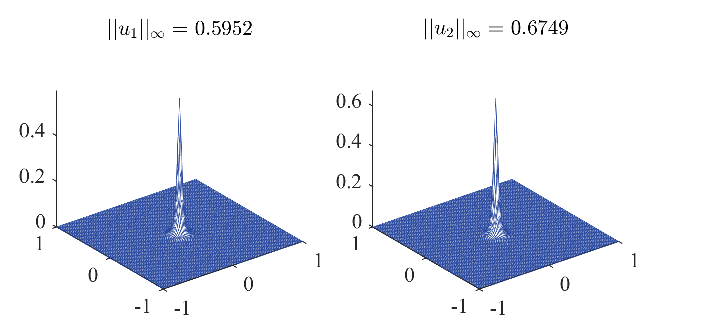}}\\
    \vspace{-.5ex}
 		\caption{The profiles (left) and surface plots (right) of the ground state solution to the coupled system~\eqref{eq:model} under the harmonic oscillator potential~\eqref{eq:poten_2} when $ g_{11} = 1$, $g_{12} = 10$ and $g_{22} = 3$, with different small $\varepsilon_i$.}
 		\label{fig:perturbed1}
 \end{figure}

\appendix
\renewcommand{\thesection}{\Alph{section}}

\section{Proof of \texorpdfstring{Lemma~\ref{lem:Hgrad-vs-Rgrad-norm}}{Lemma~{ref{lem:Hgrad-vs-Rgrad-norm}}}}\label{sec:proof-lem:Hgrad-vs-Rgrad-norm}

For all $\mathbf{u}\in\mathcal{N}$, the $H$-orthogonal projection \eqref{eq:Rgrad} implies immediately $\left(\nabla G(\mathbf{u}),\nabla_{\mathcal{N}}E(\mathbf{u})\right)_H=0$ and $\left\|\nabla_{\mathcal{N}}E(\mathbf{u}) \right\|_H^2\leq \left\| \nabla E(\mathbf{u})\right\|_H^2$. The left estimate in \eqref{eq:Hgrad-vs-Rgrad-norm} is obtained, and the rest is to verify the right estimate in \eqref{eq:Hgrad-vs-Rgrad-norm}. Taking the $H$-inner product with $\mathbf{u}$ on both sides of \eqref{eq:Rgrad} and noticing $(\nabla E(\mathbf{u}), \mathbf{u})_H =\langle E'(\mathbf{u}),\mathbf{u} \rangle= 0$, one gets
    \begin{align}\label{eq:Rgrad-u}
        \left(\nabla_{\mathcal{N}}E(\mathbf{u}), \mathbf{u} \right)_H = - \frac{(\nabla E(\mathbf{u}),\nabla G(\mathbf{u}))_H}{\|\nabla G(\mathbf{u})\|_H^2} \left( \nabla G(\mathbf{u}), \mathbf{u}\right)_H.
    \end{align}
Noting that $\left( \nabla G(\mathbf{u}), \mathbf{u} \right)_H=\langle E''(\mathbf{u})\mathbf{u},\mathbf{u}\rangle=-2K(\mathbf{u})$, the combination of \eqref{eq:Rgrad} and \eqref{eq:Rgrad-u} implies
\begin{align}\label{eq:Hgrad-vs-Rgrad}
  \nabla E(\mathbf{u}) =\nabla_{\mathcal{N}}E(\mathbf{u}) + \frac{\left(\nabla_{\mathcal{N}}E(\mathbf{u}), \mathbf{u} \right)_H}{2K(\mathbf{u})} \nabla G(\mathbf{u}).
\end{align}
To estimate the right-hand side of \eqref{eq:Hgrad-vs-Rgrad}, we denote $\bm{\phi}:=\nabla G(\mathbf{u})=(\phi_1,\phi_2,\ldots,\phi_m)\in H$. By Cauchy-Schwarz inequality, H\"older inequality and Sobolev embedding inequalities, one has
\begin{align*}
\|\nabla G(\mathbf{u})\|_H^2
&=\langle G'(\mathbf{u}), \bm{\phi} \rangle 
= 2 \int_{\Omega }\sum_{i=1}^m \left(\varepsilon_i\nabla u_{i} \cdot \nabla \phi_{i} + a_iu_{i} \phi_{i} \right) d\mathbf{x} - 4\int_{\Omega } \sum_{i,j=1}^m  g_{ij}u_{j}^2u_{i} \phi_{i} d\mathbf{x} \\
&\leq 2 \|\mathbf{u}\|_H\|\bm{\phi}\|_H + 2 \|\mathbf{a}\|_{\infty}\sum_{i=1}^m \|u_{i}\|_{L^2} \|\phi_{i}\|_{L^2} + 4\|\mathbf{g}\|_{\infty} \sum_{i,j=1}^m  \|u_{j}\|_{L^4}^2\|u_{i}\|_{L^4} \|\phi_{i}\|_{L^4} \\
& \leq 2 \|\mathbf{u}\|_H\|\bm{\phi}\|_H + C_1\|\mathbf{a}\|_{\infty}\|\mathbf{u}\|_H\|\bm{\phi}\|_H + C_2\|\mathbf{g}\|_{\infty}\|\mathbf{u}\|_H^3\|\bm{\phi}\|_H,
\end{align*}
where $C_1,C_2>0$ depend on parameters $\{\varepsilon_i\}$ and the constants in Sobolev embeddings $H_0^1(\Omega) \hookrightarrow L^2(\Omega)$ and $H_0^1(\Omega) \hookrightarrow L^4(\Omega)$, respectively. It follows that
\begin{align}\label{eq:DGnorm-by-u}
\|\nabla G(\mathbf{u})\|_H = \|\bm{\phi}\|_H \leq \tilde{C} \left(\|\mathbf{u}\|_H +\|\mathbf{u}\|_H^3 \right),
\end{align}
with $\tilde{C} = \max\{ 2+C_1\|\mathbf{a}\|_{\infty}, C_2\|\mathbf{g}\|_{\infty} \}$. 
Then, from \eqref{eq:Hgrad-vs-Rgrad}, \eqref{eq:DGnorm-by-u} and the fact that $4E(\mathbf{u})=K(\mathbf{u})\geq c_1 \|\mathbf{u}\|_H^2$, one obtains
\begin{align*}
  \left\|\nabla E(\mathbf{u}) \right\|_H
\leq \left\|\nabla_{\mathcal{N}}E(\mathbf{u})\right\|_H\left(1+\frac{\|\mathbf{u}\|_H\left\|\nabla G(\mathbf{u}) \right\|_H}{2K(\mathbf{u})} \right) 
  \leq \left\|\nabla_{\mathcal{N}}E(\mathbf{u})\right\|_H\left(1 + \frac{\tilde{C}}{2c_1} +\frac{2\tilde{C}}{c_1^2}E(\mathbf{u}) \right).
\end{align*}
Taking $C = \max\big\{1 + \tilde{C}/(2c_1),{2\tilde{C}}/{c_1^2}\big\}$ completes the proof.

\section{Sine pseudospectral discretization}\label{sec:Discre}

Here, we present the details of the sine pseudospectral discretization. Without loss of generality, we consider the 2D rectangle domain $\Omega = (-L,L)^2$, $L>0$. Let $M$ be an even number and $h=2L/M$ be the mesh size. Define the grid points $\{(x_k,y_l)\}_{k,l=0}^M$, where $x_k = -L+kh$ and $y_l = -L+lh$ ($k,l=0,1,\ldots,M$). Let ${u}_{kl}$ be the numerical approximation of the function ${u}$ at $(x_k, y_l)$. The sine pseudospectral approximation of $-\Delta{u}$ is given by
\begin{align}
    (-\Delta_h {u})_{kl} = \sum_{p,q=1}^{M-1} \left(\left(\varrho_p^x\right) ^2+ \left(\varrho_q^y\right)^2 \right)\widehat{{u}}_{pq}\sin{\frac{p k\pi}{M}}\sin{\frac{q l\pi}{M}},
\end{align}
where $\varrho_p^x = \frac{p \pi}{2L},\; \varrho_q^y = \frac{q \pi}{2L}$, and $\widehat{{u}}_{pq}$ are coefficients of the discrete sine transform (DST) of ${u}_{kl}$, i.e.,
\begin{align}
    \widehat{{u}}_{pq} = \frac{4}{M^2}\sum_{k,l=1}^{M-1} {u}_{kl} \sin{\frac{p k\pi}{M}}\sin{\frac{q l\pi }{M}}.
\end{align}

The discrete versions of relevant quantities in Algorithms~\ref{al:RNAG}-\ref{al:modified RNAG} are described as below. 
\begin{itemize}
    \item For $\mathbf{u} = (u_1,u_2,\ldots,u_m)$ and $\mathbf{v} = (v_1,v_2,\ldots,v_m)$, the discrete $H$-inner product and discrete $H$-norm are given as
\begin{align}\label{eq:discrete-norm}
(\mathbf{u},\mathbf{v})_h=h^2 
    \sum_{i=1}^m \sum_{k,l=1}^{M-1}  \varepsilon_i(-\Delta_h u_i )_{kl}(v_i)_{kl},\quad \|\mathbf{u}\|_h =  \sqrt{(\mathbf{u},\mathbf{u})_h}.
\end{align}
\item The discrete Nehari retraction of $\mathbf{u} \in \mathcal{N}$ and $\boldsymbol{\xi} \in T_{\mathbf{u}}\mathcal{N}$ is implemented based on the explicit expression \eqref{eq:Nehari retraction} with the following discrete functionals $K_h(\mathbf{v})$ and $I_h(\mathbf{v})$ at $\mathbf{v} =\mathbf{u}+\boldsymbol{\xi}$:
\begin{align*}
    K_h(\mathbf{v}) =  h^2 
    \sum_{i=1}^m \sum_{k,l=1}^{M-1}  \left(\varepsilon_i(-\Delta_h v_i )_{kl}(v_i)_{kl} + (a_i)_{kl} (v_i)_{kl}^2 \right), \quad
    I_h(\mathbf{v}) = h^2{\sum_{i,j=1}^m\sum_{k,l=1}^{M-1}  g_{ij} (v_i)_{kl}^2 (v_j)_{kl}^2 },
\end{align*}
where $(a_i)_{kl}=a_i(x_k,y_l)$. 
\item The discrete RSD direction is implemented based on \eqref{eq:eta-SD} with the discrete inner product \eqref{eq:discrete-norm} and the explicit sine pseudospectral solutions to Poisson equations in \eqref{eq: R_grad_N}:
\begin{align*}
    (\psi_i)_{kl} &= (u_i)_{kl} + \sum_{p,q=1}^{M-1} \left(\frac{(\widehat{a_iu_i})_{pq} -\sum_{j=1}^m g_{ij}(\widehat{u_j^2u_i})_{pq}}{\varepsilon_i\left((\varrho_p^x)^2 + (\varrho_q^y)^2\right)}\sin{\frac{p k\pi}{M}}\sin{\frac{q l\pi}{M}}\right),\quad i=1,2,\ldots,m, \\
(\varphi_i)_{kl} &= 2(u_i)_{kl} + \sum_{p,q=1}^{M-1} \left(\frac{2(\widehat{a_iu_i})_{pq} - 4\sum_{j=1}^m  g_{ij}(\widehat{u_j^2u_i})_{pq}}{\varepsilon_i\left((\varrho_p^x)^2 + (\varrho_q^y)^2\right)}\sin{\frac{p k\pi}{M}}\sin{\frac{q l\pi}{M}}\right), \quad i=1,2,\ldots,m.
\end{align*}
\end{itemize}
It is noted that the (inverse) DST can be implemented by fast Fourier transform (FFT) with the $O(M^2\log M)$ computational complexity.

\footnotesize
\bibliographystyle{abbrv}
\bibliography{ref}

\begin{thebibliography}{10}

\bibitem{Absil2008}
P.-A. Absil, R.~Mahony, and R.~Sepulchre.
\newblock {\em Optimization Algorithms on Matrix Manifolds}.
\newblock Princeton University Press, Princeton, NJ, 2008.

\bibitem{PRL_Akhmediev1999}
N.~Akhmediev and A.~Ankiewicz.
\newblock Partially coherent solitons on a finite background.
\newblock {\em Phys. Rev. Lett.}, 82:2661--2664, 1999.

\bibitem{ANTONIO_NLSE_jlms2007}
A.~Ambrosetti and E.~Colorado.
\newblock {Standing waves of some coupled nonlinear {Schr\"odinger} equations}.
\newblock {\em J. Lond. Math. Soc.}, 75(1):67--82, 2007.

\bibitem{Bao2004}
W.~Bao.
\newblock Ground states and dynamics of multicomponent {Bose--Einstein}
  condensates.
\newblock {\em Multiscale Model. Simul.}, 2(2):210--236, 2004.

\bibitem{2013BaoBEC}
W.~Bao and Y.~Cai.
\newblock Mathematical theory and numerical methods for {Bose-Einstein}
  condensation.
\newblock {\em Kinet. Relat. Models.}, 6:1--135, 2013.

\bibitem{Bao-Du2004}
W.~Bao and Q.~Du.
\newblock Computing the ground state solution of {B}ose--{E}instein condensates
  by a normalized gradient flow.
\newblock {\em SIAM J. Sci. Comput.}, 25(5):1674--1697, 2004.

\bibitem{Wang_Birfu2010}
T.~Bartsch, N.~Dancer, and Z.-Q. Wang.
\newblock A {L}iouville theorem, a-priori bounds, and bifurcating branches of
  positive solutions for a nonlinear elliptic system.
\newblock {\em Calc. Var. Partial Differ. Equ.}, 37:345–361, 2010.

\bibitem{Berestycki1983I}
H.~Berestycki and P.~L. Lions.
\newblock Nonlinear scalar field equations. {I}. {E}xistence of a ground state.
\newblock {\em Arch. Ration. Mech. Anal}, 82(4):313–345, 1983.

\bibitem{Boumal2023}
N.~Boumal.
\newblock {\em An Introduction to Optimization on Smooth Manifolds}.
\newblock Cambridge University Press, Cambridge, UK, 2023.

\bibitem{1994Infinite}
K.-C. Chang.
\newblock {\em Infinite Dimensional Morse Theory and Multiple Solution
  Problems}.
\newblock Birkh\"auser, Boston, MA, 1993.

\bibitem{ChenHu2008}
C.-N. Chen and X.~Hu.
\newblock Stability criteria for reaction-diffusion systems with skew-gradient
  structure.
\newblock {\em Comm. Partial Differ. Equ.}, 33(2):189--208, 2008.

\bibitem{CHEN_JCP2023}
H.~Chen, G.~Dong, W.~Liu, and Z.~Xie.
\newblock Second-order flows for computing the ground states of rotating
  {Bose-Einstein} condensates.
\newblock {\em J. Comput. Phys.}, 475:111872, 2023.

\bibitem{CHEN2010}
X.~Chen and J.~Zhou.
\newblock A local min-max-orthogonal method for finding multiple solutions to
  noncooperative elliptic systems.
\newblock {\em Math. Comp.}, 79(272):2213--2236, 2010.

\bibitem{CHEN2008}
X.~Chen, J.~Zhou, and X.~Yao.
\newblock A numerical method for finding multiple co-existing solutions to
  nonlinear cooperative systems.
\newblock {\em Appl. Numer. Math.}, 58(11):1614--1627, 2008.

\bibitem{CLXY2025SISC}
Z.~Chen, W.~Liu, Z.~Xie, and W.~Yi.
\newblock {Nehari} manifold optimization and its application for finding
  unstable solutions of semilinear elliptic {PDE}s.
\newblock {\em SIAM J. Sci. Comput.}, 47(4):A2098--A2126, 2025.

\bibitem{CORREIA2016}
S.~Correia.
\newblock Characterization of ground-states for a system of {M} coupled
  semilinear {S}chr\"odinger equations and applications.
\newblock {\em J. Differ. Equ.}, 260(4):3302--3326, 2016.

\bibitem{CORREIA2016JFA}
S.~Correia, F.~Oliveira, and H.~Tavares.
\newblock Semitrivial vs. fully nontrivial ground states in cooperative cubic
  {S}chr\"odinger systems with d $\geq$ 3 equations.
\newblock {\em J. Funct. Anal.}, 271(8):2247--2273, 2016.

\bibitem{DANCER2010953}
E.~N. Dancer, J.~Wei, and T.~Weth.
\newblock A priori bounds versus multiple existence of positive solutions for a
  nonlinear {Schr\"odinger system}.
\newblock {\em Ann. Inst. H. Poincar\'e Anal. Non Lin\'eaire}, 27(3):953--969,
  2010.

\bibitem{Djairo1998}
D.~G. {de Figueiredo} and J.~Yang.
\newblock Decay, symmetry and existence of solutions of semilinear elliptic
  systems.
\newblock {\em Nonlinear Anal., T.M.A.}, 33(3):211--234, 1998.

\bibitem{Jackson-PRL-1998}
B.~Jackson, J.~F. McCann, and C.~S. Adams.
\newblock Vortex formation in dilute inhomogeneous {Bose-Einstein} condensates.
\newblock {\em Phys. Rev. Lett.}, 80:3903--3906, 1998.

\bibitem{Kielhofer1974}
H.~Kielhöfer.
\newblock Stability and semilinear evolution equations in {Hilbert} space.
\newblock {\em Arch. Ration. Mech. Anal.}, 57:150--165, 1974.

\bibitem{1995differential}
S.~Lang.
\newblock {\em Differential and Riemannian Manifolds}.
\newblock Springer-Verlag, New York, 1995.

\bibitem{APG2015}
H.~Li and Z.~Lin.
\newblock Accelerated proximal gradient methods for nonconvex programming.
\newblock In {\em Adv. Neural Inf. Process. Syst.}, volume~28, page 379–387,
  2015.

\bibitem{Lin2005}
T.~Lin and J.~Wei.
\newblock Ground state of {N} coupled nonlinear {Schr\"odinger} equations in
  {$\mathbb{R}^n$}, $n\leq 3$.
\newblock {\em Comm. Math. Phys.}, 255(3):629–653, 2005.

\bibitem{LYZ2023SISC}
W.~Liu, Y.~Yuan, and X.~Zhao.
\newblock Computing the action ground state for the rotating nonlinear
  {Schr\"odinger} equation.
\newblock {\em SIAM J. Sci. Comput.}, 45(2):A397--A426, 2023.

\bibitem{Liu2017RAG}
Y.~Liu, F.~Shang, J.~Cheng, H.~Cheng, and L.~Jiao.
\newblock Accelerated first-order methods for geodesically convex optimization
  on {R}iemannian manifolds.
\newblock In {\em Adv. Neural Inf. Process. Syst.}, volume~30, page
  4868–4877, 2017.

\bibitem{MAIA2006JDE}
L.~Maia, E.~Montefusco, and B.~Pellacci.
\newblock Positive solutions for a weakly coupled nonlinear {Schr\"odinger}
  system.
\newblock {\em J. Differ. Equ.}, 229(2):743--767, 2006.

\bibitem{manakov1974theory}
S.~V. Manakov.
\newblock On the theory of two-dimensional stationary self-focusing of
  electromagnetic waves.
\newblock {\em Sov. Phys. JETP}, 38(2):248--253, 1974.

\bibitem{1960On}
Z.~Nehari.
\newblock On a class of nonlinear second order differential equations.
\newblock {\em Trans. Amer. Math. Soc.}, 95:101--123, 1960.

\bibitem{1961Characteristic}
Z.~Nehari.
\newblock Characteristic values associated with a class of nonlinear
  second-order differential equations.
\newblock {\em Acta Math.}, 105:141--175, 1961.

\bibitem{NAG}
Y.~Nesterov.
\newblock A method of solving a convex programming problem with convergence
  rate {$O(1/k^2)$}.
\newblock {\em Sov. Math. Dokl.}, 27(2):372--376, 1983.

\bibitem{Palais1963}
R.~S. Palais.
\newblock Morse theory on {Hilbert} manifolds.
\newblock {\em Topology}, 2(4):299--340, 1963.

\bibitem{Rabinowitz1986}
P.~H. Rabinowitz.
\newblock {\em Minimax Methods in Critical Point Theory with Applications to
  Differential Equations}.
\newblock Number~65 in CBMS Reg. Conf. Ser. Math. AMS, Providence, RI, 1986.

\bibitem{Sato2015}
Y.~Sato and Z.-Q. Wang.
\newblock Multiple positive solutions for {Schr\"odinger} systems with mixed
  couplings.
\newblock {\em Calc. Var. Partial Differ. Equ.}, 54:1373–1392, 2015.

\bibitem{Sirakov2007}
B.~Sirakov.
\newblock Least energy solitary waves for a system of nonlinear
  {Schr\"{o}dinger} equations in $\mathbb{R}^n$.
\newblock {\em Commun. Math. Phys.}, 271:199--221, 2007.

\bibitem{SOAVE2015}
N.~Soave.
\newblock On existence and phase separation of solitary waves for nonlinear
  {S}chr{\"o}dinger systems modelling simultaneous cooperation and competition.
\newblock {\em Calc. Var. Partial Differ. Equ.}, 53:689–718, 2015.

\bibitem{SOAVE2016}
N.~Soave and H.~Tavares.
\newblock New existence and symmetry results for least energy positive
  solutions of {Schr\"odinger} systems with mixed competition and cooperation
  terms.
\newblock {\em J. Differ. Equ.}, 261(1):505--537, 2016.

\bibitem{SuWeijie-2016}
W.~Su, S.~Boyd, and E.~J. Cand{{\`e}}s.
\newblock A differential equation for modeling {N}esterov's accelerated
  gradient method: Theory and insights.
\newblock {\em J. Mach. Learn. Res.}, 17(153):1--43, 2016.

\bibitem{2010Themethod}
A.~Szulkin and T.~Weth.
\newblock The method of {Nehari} manifold.
\newblock In {\em Handbook of Nonconvex Analysis and Applications}, pages
  597--632. International Press, Boston, 2010.

\bibitem{shabat1972exact}
V.~E. Zakharov and A.~B. Shabat.
\newblock Exact theory of two-dimensional self-focusing and one-dimensional
  self-modulation of waves in nonlinear media.
\newblock {\em Sov. Phys. JETP}, 34(1):62--69, 1972.

\bibitem{Hongchao2004A}
H.~Zhang and W.~W. Hager.
\newblock A nonmonotone line search technique and its application to
  unconstrained optimization.
\newblock {\em SIAM J. Optim.}, 14(4):1043--1056, 2004.

\bibitem{Zhang_pmlr2018}
H.~Zhang and S.~Sra.
\newblock An estimate sequence for geodesically convex optimization.
\newblock In {\em Proceedings of the 31st Conference on Learning Theory},
  volume~75, pages 1703--1723, 2018.

\bibitem{YuBo2020}
X.~Zhang, J.~Zhang, and B.~Yu.
\newblock Finding multiple solutions to elliptic systems with polynomial
  nonlinearity.
\newblock {\em Numer. Meth. Partial. Differ. Equ.}, 36(5):1074--1097, 2020.

\end{thebibliography}

\end{document}